\documentclass[11pt]{article}
\usepackage{amsmath}
\usepackage{hyperref}
\usepackage{pdfsync}
\usepackage{dsfont}
\usepackage{color}
\usepackage{esint}
\usepackage{amsthm}
\usepackage{enumerate}
\usepackage{amsfonts}
\usepackage{amssymb}
\usepackage{mathtools}
\usepackage{braket}
\usepackage{bm}
\usepackage{amsmath}
\usepackage{graphicx}
\usepackage{caption}
\usepackage{fullpage}
\usepackage{todonotes}
\usepackage{subcaption}
\usepackage{mathrsfs}

\usepackage[font=footnotesize,labelfont=bf]{caption}

\usepackage{mathtools}
\numberwithin{equation}{section}

\usepackage{amsfonts}
\usepackage{graphicx}
\providecommand{\U}[1]{\protect\rule{.1in}{.1in}}
\newtheorem{theorem}{Theorem}[section]

\newtheorem{definition}[theorem]{Definition}
\newtheorem{example}[theorem]{Example}

\newtheorem{proposition}[theorem]{Proposition}

\allowdisplaybreaks[3]

\newcommand{\R}{\mathbb{R}}

\newcommand{\e}{\varepsilon}

\newcommand{\changed}{black}
\newcommand{\moreedits}{black}

\title{Eikonal depth: an optimal control approach to statistical depths}
\author{Martin Molina-Fructuoso and Ryan Murray}
\date{\today}
\begin{document}
\maketitle

\begin{abstract}
	Statistical depths provide a fundamental generalization of quantiles and medians to data in higher dimensions. This paper proposes a new type of globally defined statistical depth, based upon control theory and eikonal equations, which measures the smallest amount of probability density that has to be passed through in a path to points outside the support of the distribution: for example spatial infinity. This depth is easy to interpret and compute, {\color{\changed}expressively captures multi-modal behavior}, and extends naturally to data that is non-Euclidean. We prove various properties of this depth, and provide discussion of computational considerations. In particular, we demonstrate that this notion of depth is robust under an {\color{\changed}aproximate isometrically} constrained adversarial model, a property which is not enjoyed by the Tukey depth. Finally we give some {\color{\changed}illustrative} examples {\color{\changed}in the context of two-dimensional mixture models and MNIST}. 
\end{abstract}

\section{Introduction}

In univariate statistics, quantiles, depths, and medians serve as an important cornerstone for constructing robust statistical orderings. Generalizing these notions to multivariate statistics has been the subject of significant study. Indeed, constructing multivariate notions of depth which are simultaneously robust, computable, interpretable, and statistically {\color{\changed}useful} remains challenging.

The most classical notion of multivariate depth is the {\color{\changed}\emph{Tukey}} or {\color{\changed}\emph{Halfspace depth}}, which, given a {\color{\changed}distribution $F$} over $\R^d$, is defined by
\[
{\color{\changed}  d_T(x) := \inf_{a \in \mathbb{R}^d, \,a \neq 0} F( \{ y \in \mathbb{R}^d\, | \, (x-y)\cdot a \geq 0\}).}
\]
This notion of depth, which is a natural extension of the definition in $\R$, enjoys many {\color{\changed}useful} properties. The (super) level sets are convex and nested, are invariant under affine transformations, and are robust under specific classes of changes to {\color{\changed}$F$}. Indeed the Tukey depth is considered so fundamental that {\color{\changed}its properties have almost directly been used to define what {\color{\changed}it means to be} a statistical depth \cite{zuo2000general}}.

However, the halfspace depth has some inconvenient properties. {\color{\changed}It is challenging to compute in higher dimension}. It also is difficult to extend to data which is not inherently Euclidean, and gives counter-intuitive results for distributions that are multimodal or have non-convex level sets. Finally, the notion of robustness used in studying the Tukey depth is very classical, and under other notions of adversarial robustness the Tukey depth actually turns out to be non-robust (see Section \ref{sec:robustness} for details).

The purpose of this work is to propose a new notion of depth, which we call the eikonal depth, designed to address some of these shortcomings. This depth is based upon a different extension of the one-dimensional quantile depth to higher dimension, and has a direct, interpretable formulation based upon optimal control. Like the case of the halfspace depth, our definition is global and geometric in nature. However, unlike the case of the halfspace depth, this depth is much more amenable to computation, is robust under a natural class of adversarial perturbations to inputs, and is directly extensible to non-Euclidean and discrete settings. Indeed, we argue that this notion of depth is natively understood in terms of metric geometry, as opposed to convex geometry: see Section \ref{sec:robustness}. Although there are many notions of depth \textcolor{\changed}{(some of which we describe in Section \ref{sec:literature})}, we believe that the eikonal depth provides a unique combination of features not {\color{\changed}attainable by} other {\color{\changed}curently available} notions of depth.

The remainder of the work is organized as follows: in Section \ref{sec:eikonaldef} we give the definition of eikonal depth in th Euclidean setting, and describe the connection with the classical one-dimensional depth. In Section \ref{sec:literature} we provide a review of related literature. {\color{\changed}In Section \ref{sec:mainproperties} we describe some of the main properties of the eikonal depth in Euclidean space and in Section \ref{sec:robustness} we identify a new type of adversarial robustness inherent to our definition. Section \ref{sec:extensionNonEucli} describes natural extensions of the definition to a wide class of metric spaces such as manifolds with boundary and graphs.} Finally, Section \ref{sec:numerics} describes numerical approaches for this depth, and gives a number of illustrative examples.

\subsection{Proposed model: an eikonal depth}
\label{sec:eikonaldef}
We will start by defining the eikonal depth in the population setting {\color{\changed}and will provide an analogous approach for empirical measures in Section \ref{sec:extensionNonEucli}}. {\color{\changed}As such, we will a assume that $F$ has an associated continuous density function $\rho$.} For now, we will focus our attention on two canonical cases: distributions with probability densities whose support is all of $\mathbb{R}^d$ and distributions whose support is closed and bounded. Considering densities with support on all of $\mathbb{R}^d$ allows us to include standard probability distributions such as mixtures of Gaussians, depicted in Figure \ref{fig:figure1.2}. An example of a distribution with a density with compact support is a uniform distribution on a square, as in Figure \ref{fig:figure1.1}.  
\begin{definition}
	\label{def:eikonalpaths}
	
	Let $F$ be a probability distribution with density $\rho$ and let $\phi$ be a continuous and non decreasing function $\mathbb{R}^+ \rightarrow \mathbb{R}^+$ such that $\phi(0)=0$.
	
	 We define $\mathcal{U}_x$ to be the set of continuous curves of locally finite length from $[0,\infty) \to \R^d$ if $\mbox{supp} \, (\rho) = \mathbb{R}^d$ (alternatively, from $[0,T]  \rightarrow \mathbb{R}^d$ if $\mbox{supp} \, (\rho) = \overline{\Omega}$ for some open and bounded $\Omega \subset \mathbb{R}^d$) so that if $\gamma \in \mathcal{U}_x$ then $\gamma(0) = x$ and 
		\begin{equation}
			\begin{array}{llcl}
				&	\lim_{t \to \infty} \gamma(t) = \infty & \mbox{ if } & \mbox{supp} \, (\rho) = \mathbb{R}^d, \\
				\mbox{alternatively, }&	\gamma(T) \in \partial \Omega & \mbox{ if } & \mbox{supp} \, (\rho) = \overline{\Omega} \mbox{.}
			\end{array}
		\end{equation}
		
		We then define the $\phi$-eikonal depth of a point $x$ in $F$, denoted by 
		$D_{eik}(x,F)$, by the minimization problem
			\begin{equation}
				D_{eik}(x,F):=\inf_{\gamma \in \mathcal{U}_x} \int_{\gamma} \phi(\rho ) ds  =: \inf_{\gamma \in \mathcal{U}_x} J(\gamma),
				\label{eq:energyEikonal}
			\end{equation} 
			where $\int_{\gamma} \phi(\rho) ds$ stands for the integral of the function $\phi(\rho)$ along the curve $\gamma$. In particular, in the case where the support of $\rho$ is all of $\R^d$ then this integral takes the form $\int_\gamma \phi(\rho)\,ds = \int_0^\infty \phi(\rho(\gamma(t)))|\dot \gamma(t)| \,dt$, where $\dot \gamma$ may either be a classical derivative (when available), or may be interpreted as a measure, which always is possible for curves of locally finite length.

			For the sake of clarity, when the role of $\phi$ is not important, we refer to $D_{eik}(x,F)$ simply as the eikonal depth. When $\phi(s) = s$ we call $D_{eik}(x,F)$ the \emph{unnormalized eikonal depth}. When $\phi(s) = s^{\frac{1}{d}}$ we call $D_{eik}(x,F)$ the \emph{normalized eikonal depth}.  
			

		\end{definition}
	Definition \ref{def:eikonalpaths} has mathematical meaning even if $\phi(0) \neq 0$: indeed if $\phi \equiv 1$ then this simply becomes the problem of finding the shortest path to the boundary of the domain. In our context we restrict our attention to the case where $\phi(0) = 0$ so that we may talk about distributions with unbounded support. 	
	
		In the context of optimal control theory, we can interpret the $\phi$-eikonal depth as the minimum amount of time that a particle requires to escape to infinity (or to reach the boundary of the support, depending on the support of the density) in a field \textcolor{\changed}{with velocity equal to $\frac{1}{\phi(\rho(x))}$ or alternatively as the minimal amount of $\phi$-weighted density along a path to infinity}.
		We note that we could actually replace the infimum with a minimum in the previous definition: if $\{\gamma_n\} \subset \mathcal{U}_x$ is a sequence that achives the infimum, i.e. $\lim_{n \rightarrow \infty} J(\gamma_n)= \inf_{\gamma \in \mathcal{U}_x} J(u)$, we can take a weak limit and use the continuity of $\phi,\rho$ to obtain a $\gamma$ with $J(\gamma)= \inf_{\gamma \in \mathcal{U}_x} J(\gamma)$.
		
		%
		
		We can also see from Definition \ref{def:eikonalpaths} that $D_{eik}$ is continuous if $\rho$ and $\phi$ are continuous. In many settings, the optimal cost in this type of control problem can be characterized as a solution to a Hamilton-Jacobi equation. For example, in the case when $\rho$ has support in all of $\mathbb{R}^d$ or has compact support, we may use the following equivalent definition of the $\phi$-eikonal depth:

		\begin{definition}
			\label{def:eikonalPop}
			Let $F$ be a probability distribution in $\mathbb{R}^d$ with density $\rho$ with $\mbox{supp} \, (\rho) = \mathbb{R}^d$ (alternatively with $\mbox{supp} \, (\rho)=\overline{\Omega}$ for an open and bounded $\Omega \subset \mathbb{R}^d$), and let $\phi$ be a non decreasing continuous function $\R^+ \to \R^+$.
			
			We define the $\phi$-eikonal depth $D_{eik}(x,F)$ as the solution, in the viscosity sense (see Definition \ref{def:viscositysolution}), of the equation
			\begin{equation}
				\label{eq:eikonal}
				|\nabla u (x)|= \phi(\rho (x))
			\end{equation}
			with boundary conditions $\lim_{\|x\| \rightarrow \infty} u(x)=0$ (alternatively $u=0$ on $\partial \Omega$). 
		\end{definition}

		\begin{figure}
			\centering
			\begin{subfigure}[b]{0.49\textwidth}
				\centering
				\includegraphics[width=\textwidth]{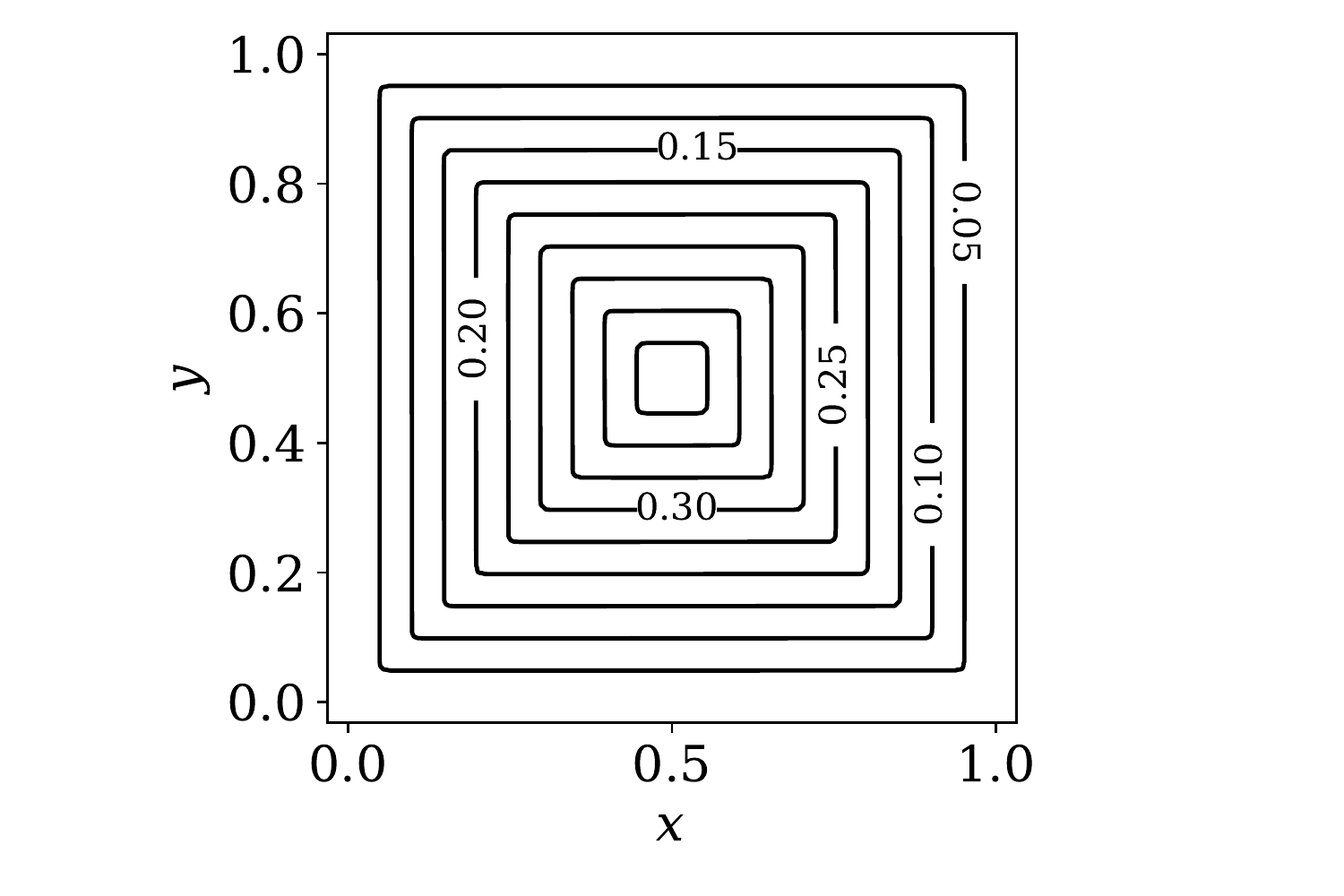}
				\caption{}
				\label{fig:figure1.1}
			\end{subfigure}
			\hfill
			\begin{subfigure}[b]{0.49\textwidth}
				\centering
				\includegraphics[width=\textwidth]{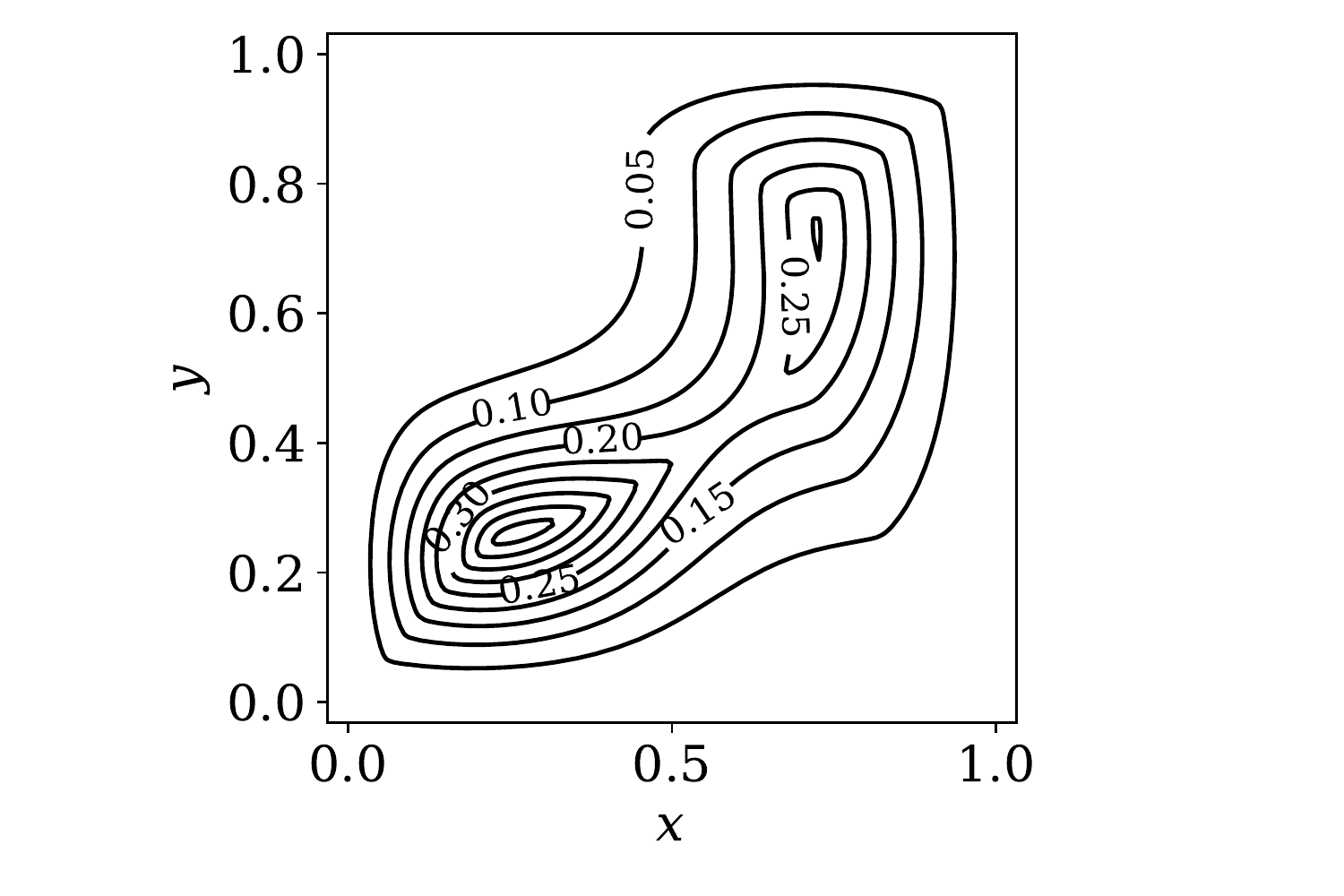}
				\caption{}
				\label{fig:figure1.2}
			\end{subfigure}
			\caption{\textcolor{\changed}{The figure shows, for $\phi(s)=s$, the level sets of the eikonal depth of a distribution whose densities have compact support (uniform density on a square in Figure \ref{fig:figure1.1}) and support on all of $\mathbb{R}^d$ (a mixture of Gaussians in Figure \ref{fig:figure1.2}). See example \ref{ex:Gaussianmixture} for more details on mixtures of Gaussian distributions.}
			}
		\end{figure}
		
		In the previous definition we refer to the ``viscosity solution'' of Equation \eqref{eq:eikonal}. The precise definition of viscosity solutions is somewhat technical \textcolor{\changed}{but a simplified version is given in Definition \ref{def:viscositysolution}}; one reference text is \cite{bardi2008optimal}, and an introduction in the context of Tukey depths is given in \cite{ryanmartin}. From a heuristic standpoint, the concept of viscosity solutions \textcolor{\changed}{provides a generalized notion of solving \eqref{eq:eikonal} at points where $u$ is not differentiable that is consistent with the control problem. Such a generalization is essential as \eqref{eq:eikonal} will usually not admit everywhere differentiable solutions}. From an algorithmic standpoint, viscosity solutions can be approximated by adding a small viscosity term ($\varepsilon \Delta D_{eik}$) to the differential equation.
		
		For the sake of clarity we give the definitions of viscosity solutions of a differential equation only in the context of Equation \ref{eq:eikonal}. More general definitions for nonlinear differential equations can be found in \cite{crandall1983viscosity}. We denote the set of continuous functions in $\Omega$ by $\mathcal{C}(\Omega)$ .
		\begin{definition}
			\label{def:viscositysolution}
			A function $u \in \mathcal{C}(\Omega)$ is a viscosity subsolution of Equation \eqref{eq:eikonal} if for every $x \in \Omega$ and every $\varphi \in \mathcal{C}^{\infty}(\Omega)$ such that $u-\varphi$ has a local maximum at $x$ in $\Omega$ we have that
			\begin{equation}
				|\nabla u (x)| \leq \phi(\rho(x)).
			\end{equation}
			
			In a similar manner, we say that a function $u \in \mathcal{C}(\Omega)$ is a viscosity supersolution of \eqref{eq:eikonal} if for every $x \in \Omega$ and every $\varphi \in \mathcal{C}^{\infty}$ such that $u-\varphi$ has a local minimum at $x$ in $\Omega$ we have that
			\begin{equation}
				|\nabla u(x)| \geq \phi(\rho(x)).
			\end{equation}
			Finally, if $u$ is both a subsolution and a supersolution of \eqref{eq:eikonal} then we call it a viscosity solution of \eqref{eq:eikonal}.
		\end{definition}

		{\color{\changed} Existence and uniqueness results for the Eikonal equation in a bounded domain $\Omega \subset \mathbb{R}^d$ are well established \cite{bardi2008optimal} and the following comparison principle is a standard result from the theory of viscosity solutions:
			\begin{proposition}
				\label{prop:boundeddomaincomp}
				Let $v$ be a supersolution of $|D u(x)|=\rho(x)$ with $u=0$ on $\partial \Omega$. If $u$ is a subsolution, then $u \leq v$.
			\end{proposition}

			For densities whose support is $\mathbb{R}^d$ we have a the following comparison principle that assumes the existence of a bounded supersolution of equation \eqref{eq:eikonal}.}
		
		\begin{proposition}
			\label{prop:comparisonunbounded}
			Let $v$ be a bounded supersolution of $|D u (x)|= \rho (x)$ with $u \rightarrow 0$ when $|x| \rightarrow \infty$. If $u$ is a subsolution, then $u \leq v$.
		\end{proposition}

		{\color{\changed}For the proof of Proposition \ref{prop:comparisonunbounded} refer to the appendix. In our case, we can straightforwardly define a bounded supersolution of \eqref{eq:eikonal} that guarantees that Proposition \ref{prop:comparisonunbounded} is applicable:
			\begin{equation}
				\label{eq:boundedsupersolution}
				v(x)= \min \left\{ \int_{-\infty}^{x_1} \rho(\xi,x') d \xi, \int_{x_1}^{\infty} \rho(\xi,x') d \xi \right\} \mbox{,}
			\end{equation} 
			where $x'$ stands for the last $d-1$ coordinates of a point $x=(x_1,\ldots,x_d)$. The inequality relating sub- and supersolutions in Propositions \ref{prop:boundeddomaincomp} and \ref{prop:comparisonunbounded} implies that the explicit supersolution \eqref{eq:boundedsupersolution} provides an upper bound for the depth.
		} 	

\subsection{One dimensional case}
\label{sec:onedimensional}

It is useful to consider how this definition relates to the classical definition of quantiles, depths, and medians in one dimension. Given a density function on the real line, we may readily define the quantile depth by the formula
\begin{displaymath}
	D_Q(x) = \min \left( \int_{-\infty}^{x} \rho(z) \,dz, \int_{x}^{\infty} \rho(z) \,dz \right)
\end{displaymath}
We may then be express this equation in many different forms. For example, we may rewrite
\begin{displaymath}
	D_Q(x) = \min_{a \neq 0} \int_{a(x-y) \leq 0} \rho(z) \,dz.
\end{displaymath}
This definition is, naturally, unnecessarily complicated in $\R$, as we only need to really check over $a \in \{-1,+1\}$. However, this definition readily extends to $\R^d$ by simply replacing the region of integration by $a\cdot (x-y) \leq 0$: this corresponds to integrating over halfspaces, and matches the definition of the classical Tukey depth.

On the other hand, in the context of the eikonal depth in one dimension, we notice that
  \begin{align*}
    D_{eik}(x) &=  \min_{\gamma \in \mathcal{U}_x} \int_{0}^{\infty} \rho(\gamma(t))|\dot \gamma(t)|\,dt. \\
    &= \min\left( \min_{\gamma \in \mathcal{U}_x, \dot \gamma \geq 0} \int_{0}^{\infty} \rho(\gamma(t))|\dot \gamma(t)|\,dt \, , \, \min_{\gamma \in \mathcal{U}_x, \dot \gamma \leq 0} \int_{0}^{\infty} \rho(\gamma(t))|\dot \gamma(t)|\,dt\right) \\
    &= \min \left( \int_{-\infty}^{x} \rho(z) \,dz, \int_{x}^{\infty} \rho(z) \,dz \right) = D_Q(x),
  \end{align*}
  where the second equality holds since we can always decrease the cost by replacing a trajectory with one that is monotone, and the third equality holds by the change of variables formula.

As is the case with the Tukey depth, such a definition is unnecessarily complex, as there are only truly two paths leading to spatial infinity along the real line. But this definition extends directly to higher dimension\textcolor{\changed}{, motivating Definition \ref{def:eikonalpaths}}.

\textcolor{\changed}{
We also notice that in one dimension we may directly verify that the quantile depth, which is equal to the eikonal depth, is a viscosity solution of the eikonal equation \eqref{eq:eikonal}, under mild continuity assumptions upon $\rho$. Indeed, the quantile depth is differentiable at any point except for on the boundary of the set $\{D_Q = 1/2\}$, and by computing the derivative, using the fundamental theorem of calculus we see that it will satisfy the equation in the classical sense at those points, which automatically implies the super- and sub-solution inequalities. On points where $D_Q$ is not differentiable, we can directly verify that the super- and sub-solution inequalities are satisfied.}

\subsection{Related literature}
\label{sec:literature}

Multivariate medians and depths have been studied within the context of robust statistics for many years: for example the Tukey depth was introduced in the mid 1900's \cite{Hotelling,hodges1955bivariate,Tukey-1975}. Many alternative notions of depth have been proposed. These include, for example, the projection depth \cite{zuo2003projection}, the Oja depth \cite{oja1983descriptive}, the zonoid depth \cite{dyckerhoff1996zonoid}, the Mahalanobis depth \cite{mahalanobis1936generalized}, the convex peeling depth \cite{barnett1976ordering}, and the Monge Kantorovich depth \cite{chernozhukov2017monge}. These depths each carry unique advantages and disadvantages, which are compared in \cite{mosler2020choosing}. In briefest summary, these depths tend to either be i) robust and intepretable, but difficult to compute (e.g. the Tukey, projection, convex peeling), ii) interpretable and computable but not as expressive or robust (e.g. the Mahalanobis depth), or iii) not as easy to interpret (e.g. the zonoid or Monge-Kantorovich depth).

Depth functions have seen application in many different settings. They help to identify inliers and outliers of distributions, which can be an important pre-processing step in many tasks \cite{zuo2003projection,jornsten2004clustering}. They provide an ordering of data, which is useful for certain types of statistical tasks \cite{wang2005nonparametric,zuo2004stahel}. For many choices of depth these quantities are robust, meaning that they are insensitive to certain types of (potentially adversarial) perturbations \cite{zuo2000general}. Depths have also been used for data visualization \cite{rousseeuw1999bagplot}.

Despite their central place in robust statistics, our understanding of many of their properties remains incomplete. For example, in the last five years there have been many works exploring analytical properties of halfspace depths, including connections to convex geometry \cite{nagy2019halfspace} and differential equations \cite{ryanmartin}. Fundamental questions about whether depths characterize their distributions, and whether depth functions are smooth and well approximated by empirical approximations, are still of current interest \cite{masse2004asymptotics,nagy2021halfspace,nagy2019halfspace}. Similar questions in the context of other depths have also been the topic of recent work \cite{calder2020limit,cook2020rates}. Some extensions of the halfspace depth to non-Euclidean settings have also been given in \cite{carrizosa1996characterization}, \cite{small1997multidimensional}.

%

Even defining what is meant by a statistical depth is not immediately obvious. The definition from one prominent work \cite{zuo2000general} is discussed in more detail in Section \ref{sec:otherdefs}. However, some of those notions are rather rigid, and have been relaxed in other works. For example, the convexity of level sets is relaxed in \cite{chernozhukov2017monge} and subsequent works. Some authors prefer depths that can capture clustering and multi-modality. This often comes by eschewing a globally defined depth for one that is locally defined: this is a sticking point for many authors that prefer to think of depths as globally defined orderings that are distinct from local densities. Our work charts a path between these two points of view, by constructing a depth which is globally defined, but which is flexible enough to capture {\color{\changed}multi-modal behavior.}

Our definition builds naturally upon a very {\color{\changed}mature} literature surrounding optimal control and Hamilton-Jacobi equations, which only very recently have been linked to statistical depths in the {\color{\changed}context of the halfspace depth in \cite{ryanmartin} and the convex hull depth in \cite{calder2020limit}}. The eikonal equation, usually with constant right hand side, is the most classical example of a viscosity solution of a Hamilton-Jacobi equation, and references on that topic may be found in \cite{bardi2008optimal,katzourakis2014introduction}. This equation also characterizes distance functions, {\color{\changed}which are a central topic in Riemannian \cite{do1992riemannian} and, more generally, metric geometry \cite{burago2001course}.} Indeed, in some ways the depth that we define here may be seen as a modest generalization of the distance function in those contexts, but the interpretation as a statistical depth is, to our knowledge novel.

Numerical solutions for Hamilton-Jacobi equations, and more specifically the eikonal equation, have been well-studied. The fast marching method \cite{sethian1999level}, which marries the dynamic programming principle with upwind numerical schemes, is a standard tool in solving these equations numerically. This method has also been extended to unstructured grids \cite{sethian2000fast} and graphs \cite{desquesnes2013eikonal}. We give a very brief treatment of these methods, with a special focus on graphs based upon empirical measures, in Section \ref{sec:numerics}.

Finally, there has been a lot of recent activity proving a rigorous connection between graph-based variational methods on graphs and their continuum limits. While many of these works have focused on graph Laplacians and total variation energies and their associated statistical problems (see, e.g.,  \cite{trillos2018error,trillos2017new,garcia2020maximum,trillos2016continuum}), \textcolor{\changed}{a few recent works have also made steps in this direction in the context of eikonal equations \cite{calder2021boundary,fadili2021limits}}. Discussion of empirical approximations of the boundary of domains, an issue we encounter in Section \ref{sec:extensionNonEucli} can also be found in \cite{calder2021boundary}, \cite{vaughn2019diffusion}. These works provide context for the scalings that we chose in Section \ref{sec:graph_setting} when constructing depths on empirical measures. Finally, while preparing this work we were made aware of another group independently working on a similar class of eikonal equations on graphs \cite{Calder-Ettehad-2022}. Their work is especially focused on cluster-aware distances, wherein $\phi$ is monotonically decreasing, but includes our framework as a particular case.

\section{Properties of eikonal depth}
\subsection{{\color{\changed}Main properties}}
\label{sec:mainproperties}
Before describing in more detail the properties of the eikonal depth, it is illustrative to consider a few examples where the depth can be explicitly calculated.
\begin{example}
	Let $\Omega$ be a bound, open, convex domain, and let $\rho \equiv \frac{1}{\mathcal{L}^d(\Omega)}$. Then the eikonal depth $u$ is given by
	\begin{displaymath}
		{\color{\changed}	D_{eik}(x)} = \phi\left(\frac{1}{\mathcal{L}^d(\Omega)}\right) d(x,\partial \Omega), \qquad d(x,\partial \Omega) := \inf_{y \in \partial \Omega} |x-y|.
	\end{displaymath}
	We notice that if $\phi(s) = s^{1/d}$ then if we replace $\Omega$ with $a\Omega$ then the leading term in the expression scales like $a^{-1}$. It is straightforward to show that the $d(x,\partial \Omega)$ term will scale like $a$ when we rescale $\Omega$, and hence when we use $\phi(s) = s^{1/d}$ then we observe that{\color{\changed}, the eikonal depth is scale invariant in this example.}
	
	In the case where $\Omega$ is given by the unit ball, we may directly compute
	\[
	{\color{\changed}	D_{eik}(x)} = (1-|x|) \phi\left( \frac{\Gamma(d/2 + 1)}{\pi^{d/2}} \right),
	\]
	{\color{\changed}where $\Gamma$ is the Gamma function.}
	
	We notice that if we let $\phi(s) = s^{1/d}$ then the power $d$ disappears from the $\pi$, and one can use Sterling's formula to approximate $\Gamma(d/2+1)^{1/d} \sim \left( \frac{d}{2e} \right)^{1/2} (\pi d)^{\frac{1}{2d}}$, which in turn implies that the maximum eikonal depth of a unit ball in $\R^d$ is well-approximated by $\left( \frac{d}{2\pi e} \right)^{1/2} (\pi d)^{\frac{1}{2d}}$. \textcolor{\changed}{This implies that for large $d$ the maximum eikonal depth of a spherical distribution behaves like $(\frac{d}{2 \pi e})^{1/2}$, which stands in constrast with the maximum of the Tukey depth of a spherical distribution $\frac{1}{2}$, which is independent of the dimension $d$.} 
\end{example}

\begin{example}
	Let $\rho$ be a multivariate normal distribution with identity covariance. We notice that we may write $\rho(x) = G(|x|)^d$, where $G$ is a standard one-dimensional gaussian.  Then we compute that the depth of $\rho$ is given by
	\[
	{\color{\changed}	D_{eik}(x)} = \int_{|x|}^\infty \phi(G(t)^d) \,dt.
	\]
	When we use $\phi(s) = s^{1/d}$, this simplifies to just the tail integral of a one-dimensional Gaussian. In this case the maximum depth is given by $1/2$ \textcolor{\changed}{independently of $d		$}.
\end{example}

{\color{\changed}A well-known property of the Tukey depth is its invariance under affine transformations, a property shared by some other statistical depths such as the Mahalanobis, simplicial, zonoid and convex peeling depths \cite{mosler2020choosing}. Although the eikonal depth is not invariant under affine transformations, an affine invariant version of it can be constructed, in the same way as for the Oja, spatial, and lens depth \cite{mosler2020choosing}. More details on this procedure are provided after Proposition  It is not hard to see that the eikonal depth is not affine invariant:} As can be seen from its control interpretation, the value of the depth depends on the length of trajectories. A given affine transformation $Ax +b$ can modify the length of a curve while keeping the velocity $\frac{1}{\rho(A^{-1} (x-b))}$ along the curve unchanged, which may result in a change in the minimum traveling time in Definition \ref{def:eikonalpaths}. {\color{\changed}The eikonal depth is, however, invariant under rigid motions:}

\begin{proposition}\label{prop:rigid}
	For any choice of $\phi$ the eikonal depth is invariant under rigid (i.e. distance preserving) affine transformations.
\end{proposition}

\begin{proof}
	The proof of the previous proposition simply rests on the fact that the cost associated with any control $u$ is unchanged by such a rigid affine transformation. 
\end{proof}
Next, we consider the effect of uniform scalings of $x$ upon the eikonal depth. We observe invariance, {\color{\changed}up to a scaling factor,} for transformations of the type $a x + c$ for $a,c \in \mathbb{R}$. 

\begin{figure}
	\centering
	\begin{subfigure}[b]{0.49\textwidth}
		\centering
		\includegraphics[width=\textwidth]{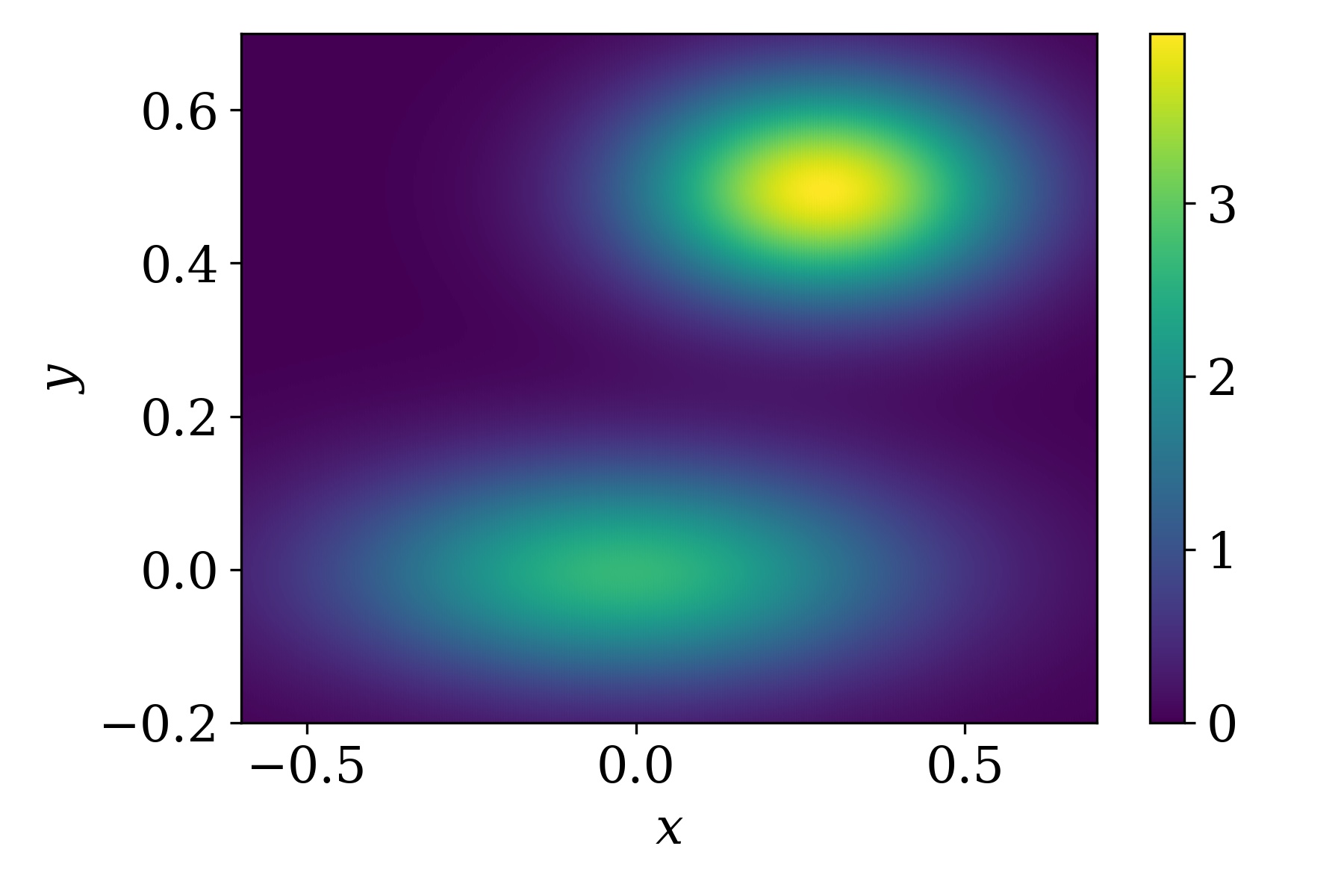}
		\caption{}
		\label{fig:densityGaussians}
	\end{subfigure}
	\hfill
	\begin{subfigure}[b]{0.49\textwidth}
		\centering
		\includegraphics[width=\textwidth]{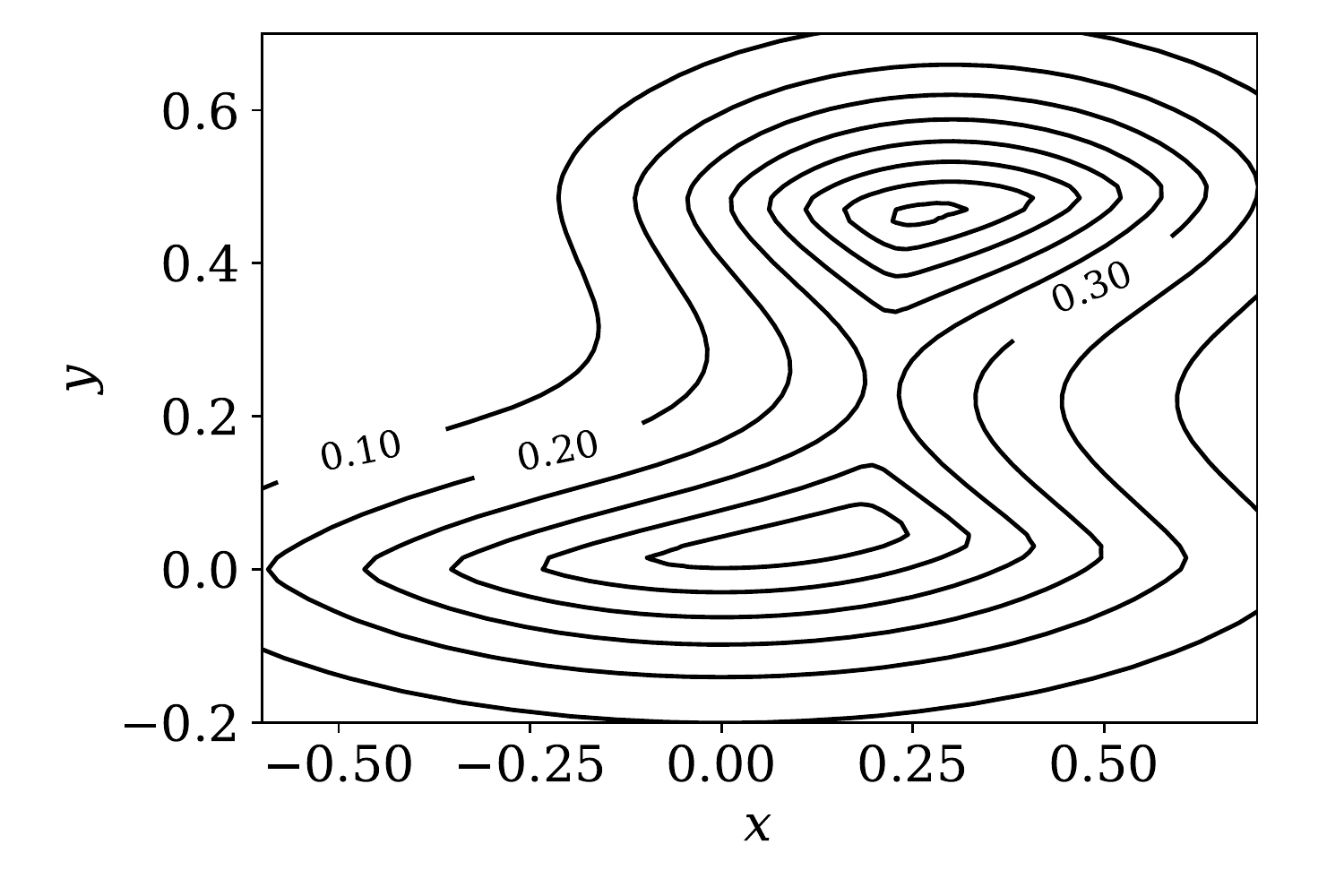}
		\caption{}
		\label{fig:contoursGaussians05}
	\end{subfigure}
	\begin{subfigure}[b]{0.49\textwidth}
		\centering
		\includegraphics[width=\textwidth]{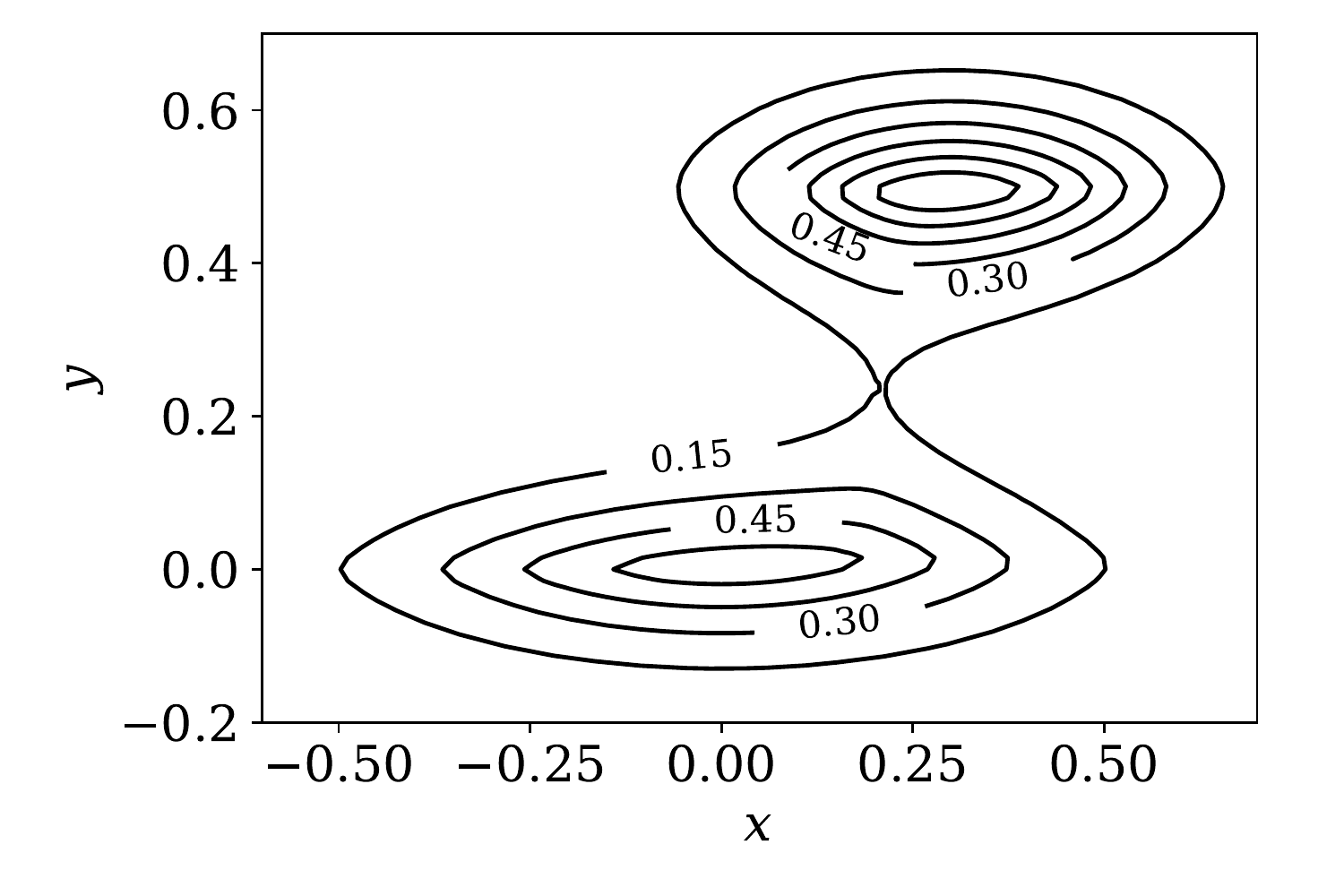}
		\caption{}
		\label{fig:contoursGaussians1}
	\end{subfigure}
	\hfill
	\begin{subfigure}[b]{0.49\textwidth}
		\centering
		\includegraphics[width=\textwidth]{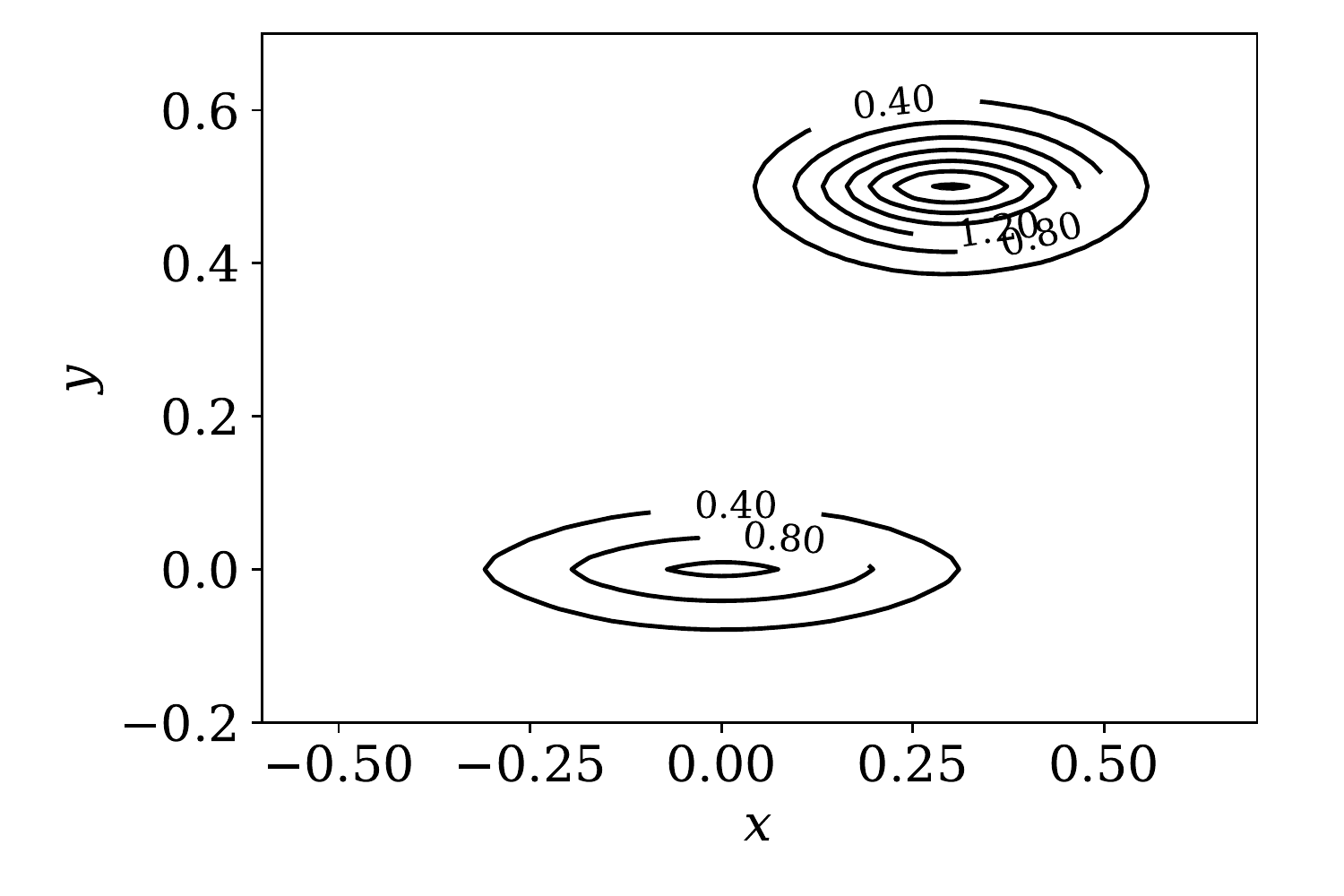}
		\caption{}
		\label{fig:contoursGaussians2}
	\end{subfigure}
	\caption{For $\phi(s)=s^{\alpha}$, the figures show the effect of $\alpha$ on the eikonal depth of a Gaussian mixture. Figure \ref{fig:densityGaussians} shows the probability density for the Gaussian mixture. Figures \ref{fig:contoursGaussians05}, \ref{fig:contoursGaussians1}, and \ref{fig:contoursGaussians2} show some level sets of the eikonal depth (for $\alpha=0.5$, $\alpha=1$, and $\alpha=2$) of the probability distribution with density as depicted in \ref{fig:densityGaussians}. 
		\label{fig:GaussiansAlpha}
	}
\end{figure}

\begin{proposition}
	For any $\phi(s) = s^\alpha$, $\alpha>0$, the eikonal depth is \emph{weakly scale invariant} in the sense that if $a^{d}\tilde \rho(ax + c) = \rho(x)$, with any $a>0$ and $c \in \R^d$, {\color{\changed}and $D_{eik}$} is the eikonal depth associated with $\rho$ (according to the function $\phi$) then the eikonal depth {\color{\changed}$\tilde D_{eik}$} associated with $\tilde \rho$ is given by $\tilde u(ax + c) = a^{-\alpha d + 1} u(x)$.
	\label{prop:scale-invariant}
\end{proposition}

\begin{proof}
	Let $\gamma:[0,\infty) \to \R^d$ be a path which begins on $x \in \Omega$ and satisfies $\lim_{t \to \infty} \gamma(t) \in \partial \Omega$. Let $\widetilde \gamma(t) = a \gamma (t)+c$. By the definition of $\widetilde{\gamma}$ we have that $\lim z(t) \in \partial \widetilde \Omega$, where $\widetilde \Omega= a \Omega +c$. The cost associated with $\widetilde{\gamma}$ is given by
	\begin{displaymath}
		\int_0^T \widetilde{\rho}^{\alpha}(\widetilde{\gamma}(t)) | \dot{\widetilde{\gamma}} (t)| \,dt = \int_0^T a^{-\alpha d}\rho^{\alpha}(\gamma(t)) a|\dot \gamma(t)|\,dt.
	\end{displaymath}
	Taking an infimum over paths on both sides gives the desired result.
\end{proof}

We notice that in the previous proposition that choosing $\alpha = 1/d$ is a critical scaling, in which the depth function is invariant under uniform scalings. \textcolor{\changed}{This property is important because a result by Serfling \cite{serfling2010equivariance} guarantees that functions with this invariant property can be made affine invariant using a scatter transform. Following \cite{mosler2020choosing}, a scatter matrix $R_X$ is a positive definite $d \times d$ matrix that satisfies $R_{A X+b}=\lambda_{X,A,b} \, A \, R_X \, A^T$ for any $A$ of full rank, any $b$, and some $\lambda_{X,A,b}>0$. In order to make the depth affine invariant, one needs to transform the data as $x \mapsto R_X^{-1/2} (x - \theta(X))$ where $R_X$ is the scatter matrix and $\theta(X)$ is a location parameter}. A computational example of this procedure applied to a mixture of two Gaussian distributions is shown in Figure \ref{fig:GaussianMixture}, where we start with two Gaussian distributions centered at the origin and $(0.3,0.5)$ with diagonal covariance matrices. \textcolor{\changed}{Using the covariance matrix of the Gaussian mixture as the scatter matrix $R_X$, and the mean of the Gaussian mixture, $\mu$, as the location parameter, we transformed the data as $x \mapsto R^{-\frac{1}{2}}(x-\mu)$.}

\begin{figure}
	\centering
	\begin{subfigure}[b]{0.49\textwidth}
		\centering
		\includegraphics[width=\textwidth]{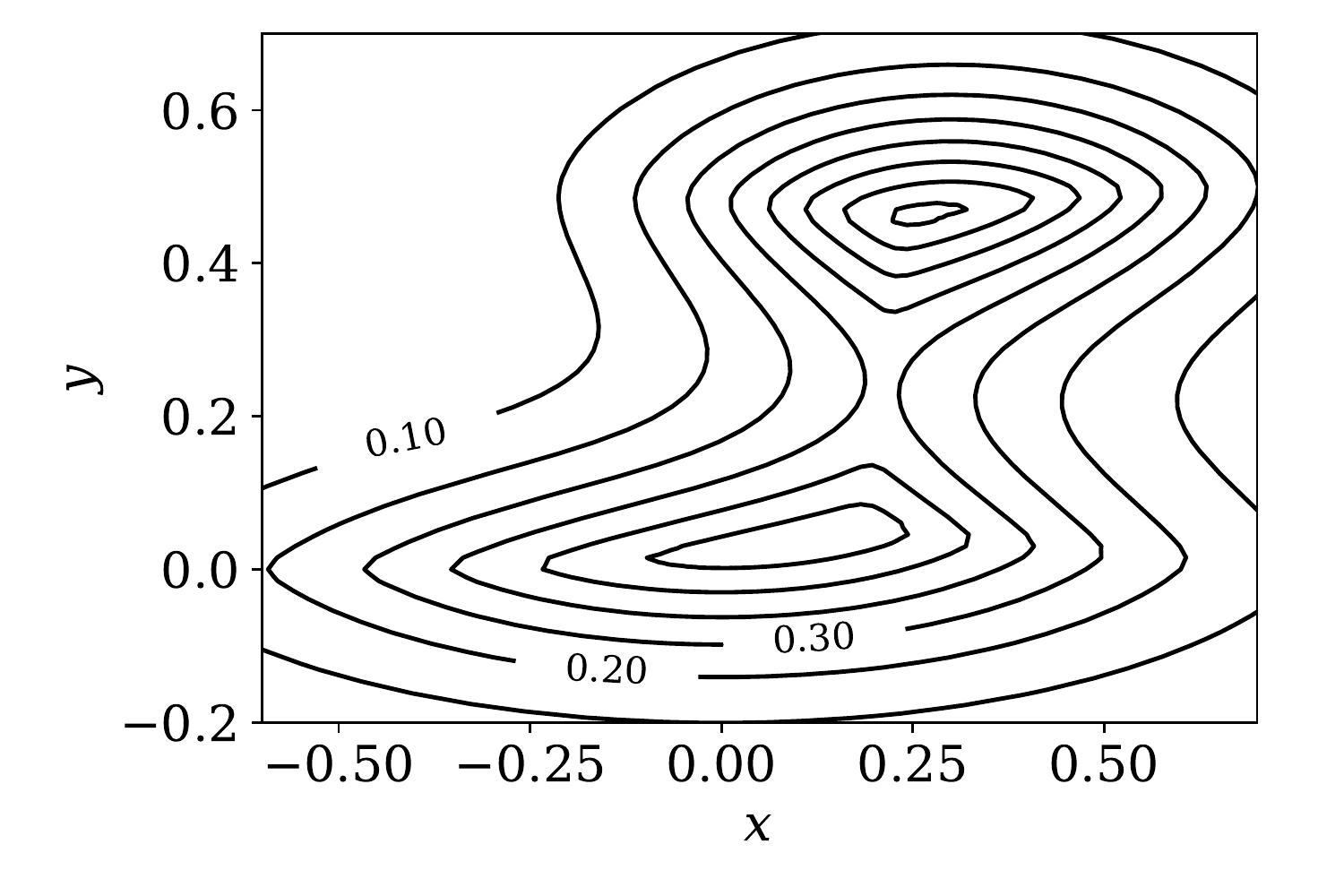}
		\caption{}
		\label{fig:GaussiansContours}
	\end{subfigure}
	\hfill
	\begin{subfigure}[b]{0.49\textwidth}
		\centering
		\includegraphics[width=\textwidth]{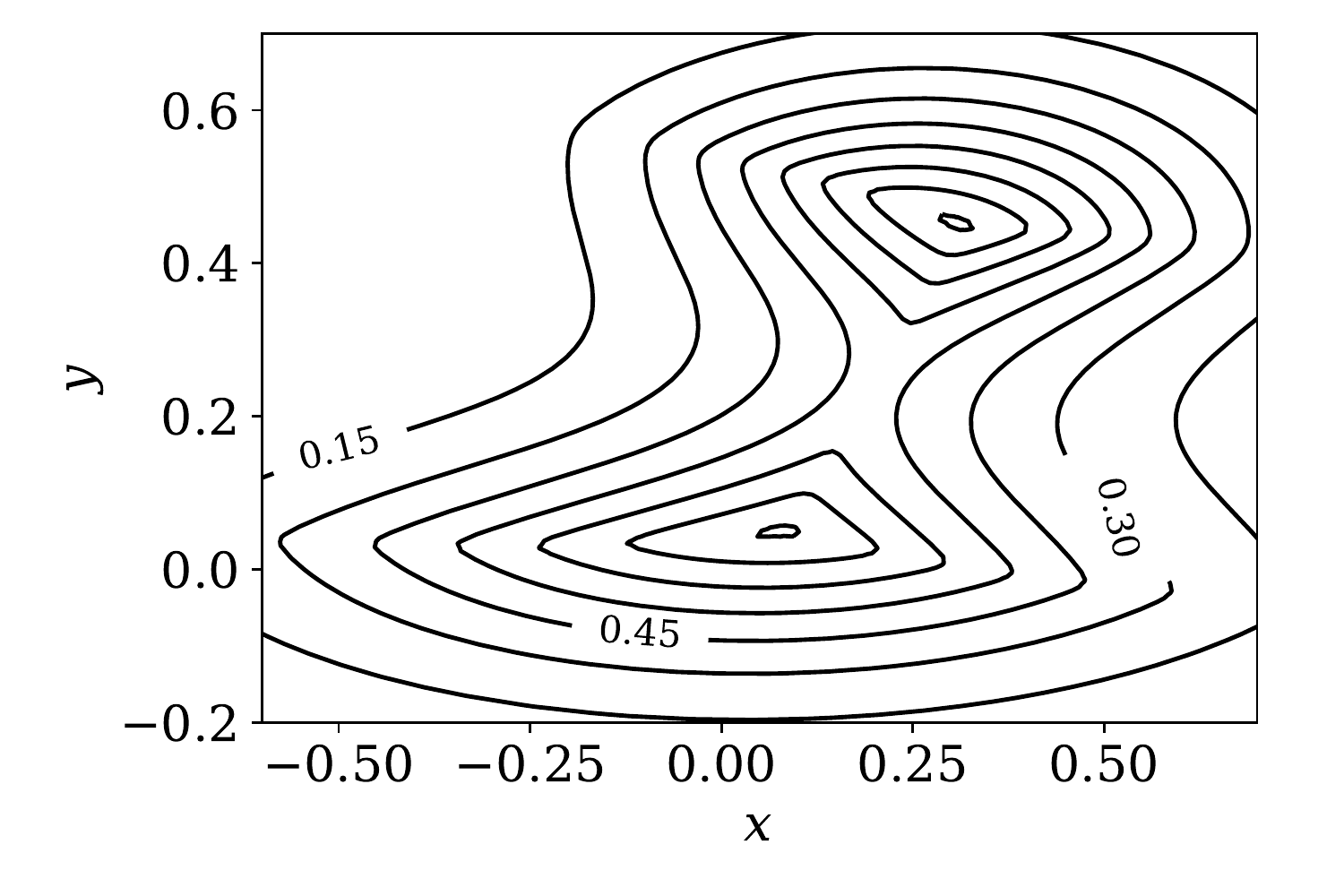}
		\caption{}
		\label{fig:GaussiansInvContours}
	\end{subfigure}
	\caption{The figure shows some level sets of the eikonal depth of a Gaussian mixture \textcolor{\changed}{with $\phi(s)=s^{1/2}$} and some level sets of the corresponding affine invariant version.
		\label{fig:GaussianMixture}
	}
\end{figure}

\textcolor{\changed}{The Tukey depth admits a uniform maximal bound independent of dimension, namely $1/2$. It is natural to consider whether eikonal depths have a similar property. }
The following proposition provides a first step in that direction.

\begin{proposition}
	Within the class of radially symmetric distributions:
	\begin{enumerate}
		\item {\color{\changed}For any $d>1$ there exist distributions with arbitrarily large maximal \emph{unnormalized eikonal depth}.}
		{\color{\changed}\item For any $d>1$ there exist distributions with arbitrarily large maximal \emph{normalized eikonal depth}.}
	\end{enumerate}
	\label{prop:median-UB}
\end{proposition}

\begin{proof}
	First, for the unnormalized depth, we notice that by simply scaling the independent variable $x$ by a factor $a$ the depth function increases by a factor of $a^{1-d}$ ({\color{\changed}see Proposition \ref{prop:scale-invariant}}). In turn, we may obtain the first result by simply scaling any radially symmetric function appropriately.
	
	\textcolor{\moreedits}{For the second, we consider truncations of the function $|x|^{-d}$. In particular, we let 
		\begin{displaymath}
		\rho_\e(x) = \begin{cases} (1-\delta) \displaystyle \frac{1}{-\ln(\e) S_{d-1}}|x|^{-d} &\text{ for } \e < |x| < 1 \\ \delta \displaystyle \frac{1}{V_d(\varepsilon)} &\text{ otherwise,} \end{cases} 
		\end{displaymath}
		where $\delta$ is such that $\rho_{\varepsilon}(x)$ is continuous at $|x|=\varepsilon$.
		Here we use $S_{d-1}$ to denote the surface area of the $d-1$ dimensional sphere and $V_d(\varepsilon)$ to denote the volume of the $d$ dimensional ball of radius $\varepsilon$. The maximal depth for this distribution is clearly attained at the origin. This value may be computed by integrating, after a suitable rotation, along the line segment described by $\gamma(t)=(t,\ldots,0)$, where $t \in [0,1]$. For the normalized eikonal depth we have 
		\begin{equation}
		D_{eik}(x)=\int_{0}^{1} \rho_{\varepsilon}^{1/d} (t) dt \geq \int_{\varepsilon}^{1} \rho_{\varepsilon}^{1/d}(t)= \int_{\varepsilon}^{1} \frac{1}{(-\ln(\e))^{1/d} S_{d-1}^{1/d} | t|} dt=\frac{(-\ln(\e))^{1-1/d}}{S_{d-1}^{1/d}}.
		\end{equation}	
		Taking $\e \to 0$ gives the desired result.}
	
\end{proof}

\textcolor{\changed}{The Tukey depth also admits a dimension dependent uniform lower bound on it's maximal value. In particular, for any distribution on $\R^d$ the point of greatest halfspace depth will have depth greater than or equal to $\frac{1}{d+1}$. The following example demonstrates that this is not the case for the eikonal depth.}

\begin{example}\label{ex:median-LB}
  For $d>1$, \textcolor{\changed}{ and for any $\phi$ which satisfies $\phi(0) = 0$, there exist distributions with arbitrarily small maximum $\phi-$eikonal depth. Let $\varphi$ be a smooth probability density compact support on $[-1/2,1/2]^d$, and let our probability distribution be given by
  \[
    \rho(x) =\frac{1}{K} \sum_{i=1}^K \varphi(x-x_i),
  \]
  where $x_i$ are distinct elements of $\mathbb{Z}^d$, namely they are points with distinct integer coordinates. We claim that the maximal Tukey depth of this distribution may be bounded by $\frac{C}{K}$. This can be proven by constructing a path from spatial $\infty$ to any point $x$ which only crosses the supports of one of the $\varphi(x-x_i)$, and then bounding the cost of that path due to $\varphi(x-x_i)$. More specifically, let $M$ be an upper bound on $\varphi$, and consider a point $x$ so that $x-x_i \in [-1/2,1/2]^d$ (other points $x$ can be handled in an analogous way). We then consider the trajectory defined by joining the ray $\gamma(t) = x_i - (1/2+t,1/2,\dots,1/2)$, $t \in [0,\infty)$, with the line segment connecting $x_i - (1/2,1/2,\dots,1/2)$ to $x$. The integrated cost due to the ray out to infinity will be zero, as the $x_i$ all have integer coordinates and $\varphi$ has support on $[-1/2,1/2]^d$. On the other hand, the cost due to the line segment connecting  $x_i - (1/2,1/2,\dots,1/2)$ to $x$ will have cost at most $\phi\left(\frac{M}{K}\right) \sqrt{d}$, as the density $\rho$ is bounded from above by $M/K$, $\phi$ is non-decreasing, and the length of the line segment to $x$ is at most $\sqrt{d}$. This proves the claim.} \end{example}

  \textcolor{\changed}{The eikonal depth does not necessarily define a single center, i.e. a point with maximal depth, in contrast to the Tukey depth and other well-studied statistical depths such as the Mahalanobis depth, or the convex peeling depth \cite{mosler2020choosing}. In some settings, this is actually a desirable property of the eikonal depth because it enables it to capture the multimodality of probability distributions, a property that is not shared by the aforementioned definitions of depth. In Section \ref{sec:otherdefs} we present a detailed discussion on the choice one faces between defining a center outward ordering and capturing the multimodality of distributions when defining a statistical depth.} {\color{\changed}For our eikonal depth, we will see} that when the distribution has several peaks that are high enough then there is a local maximum of the depth near each peak. Before formalizing this result for multimodal distributions we first need to introduce the following definition for the modes:

\begin{definition}
  We say that a distribution has well-separated modes \textcolor{\changed}{with respect to  $\phi$ if
  \begin{enumerate}
    \item Its density function has a finite number of local maxima (modes), which we write as $\{x_i\}_{i=1}^N$.
     \item There exists a finite, disjoint family of open balls $B_i$ so that each $B_i$ contains exactly one $x_i$.
     \item For all $i$ and for any  two points $x,x' \in \partial B_i$ there exists a curve $\gamma_{x,x'}$ which joins them and belongs to $\partial B_i$ so that we have $\int_{\gamma_{xx'}} \phi(\rho) ds \leq \min \left\{\int_{\tau_{xx_i}}\phi(\rho) ds, \int_{\tau_{x'x_i}}\phi(\rho) ds\right\}$ where $\tau_{yx_i}$ is an arbitrary curve joining $y \in \partial B$ and the mode $x_i$.
  \end{enumerate}
}
	\label{def:well-separated}
\end{definition}

\textcolor{\changed}{ With this definition in hand we can show the next property that provides a way of associating certain classes of modes with local maxima of the eikonal depth.} 
\begin{proposition}
	Suppose that $\rho$ has well-separated modes \textcolor{\changed}{with respect to $\phi$}, and let $B_i$ be the balls used to describe the well-separated property of the mode occuring at $x_i$. Then there is a local maximum of the eikonal depth inside $B_i$.
	\label{prop:well-separated}
\end{proposition}

\begin{proof}
  Let $p$ be an optimal trajectory, according to the energy \eqref{eq:energyEikonal},  terminating at the mode $x_i$. This path must pass through some point, which we call $\tilde x$, on the boundary of $B_i$. Now let $\hat x$ be any point on the boundary of $B_i$ and $\gamma_{\tilde x \hat x}$ be the path joining $\tilde x$ and $\hat x$ \textcolor{\changed}{from Definition \ref{def:well-separated}.} Then we have \textcolor{\changed}{by Definition \ref{def:eikonalpaths}, the well-separated property and choice of $\tilde x$, that
	\begin{displaymath}
	  D_{eik}(\hat x,F) \leq D_{eik}(\tilde x,F) + \int_{\gamma_{\tilde x \hat x}} \phi(\rho) ds \leq D_{eik}(\tilde x,F) + \inf_{\gamma \in \mathcal{U}_{\tilde x x_i}} \int_{\gamma} \phi(\rho) ds = D_{eik}(x_i,F) ,
	\end{displaymath}
	where $\mathcal{U}_{\tilde x x_i}$ denotes the set of paths starting at $\tilde x$ and ending at $x_i$.
      As $D_{eik}$ is continuous, on the set $\overline{B_i}$ it must obtain its maximal value. For any $x \in \partial B_i$ we have that $D_{eik}(x) \leq D_{eik}(x_i)$, which means that we can always find a maximizer of $D_{eik}$ in the interior of $B_i$. In turn $D_{eik}$ has a local maximizer in $B_i$. We note that such a minimizer need not be exactly $x_i$.}

	
\end{proof}

We notice that while directly proving the well-separated property may be challenging, the idea is rather intuitive: if the mode is steep enough then it should be {\color{\changed}"cheaper"} to go around than to go through. The next example illustrates one setting where such a property can be established.

\begin{figure}[h!]
	\centering
	\begin{subfigure}[b]{0.49\textwidth}
		\centering
		\includegraphics[width=\textwidth]{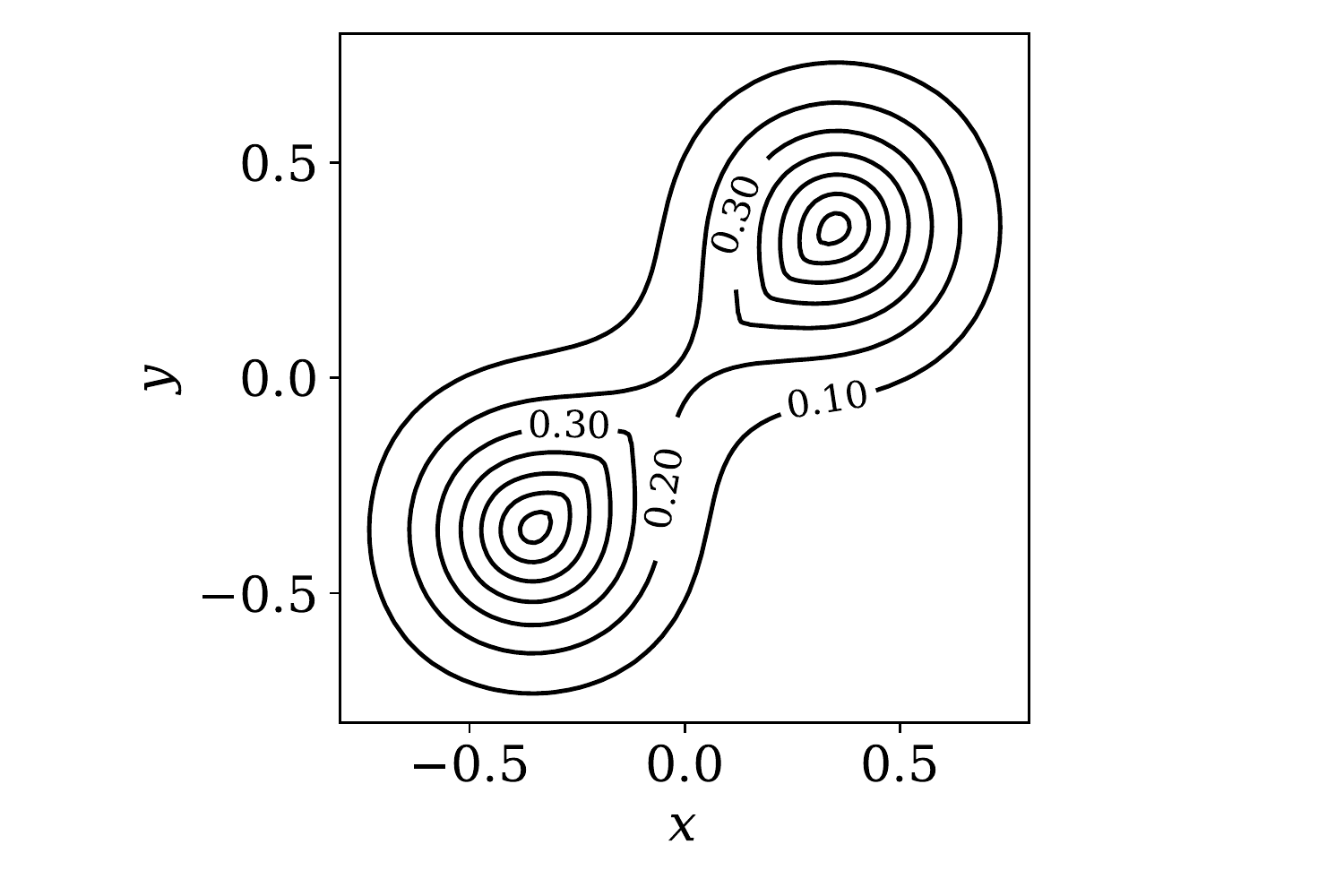}
		\caption{}
		\label{fig:Gaussians4standard}
	\end{subfigure}
	\hfill
	\begin{subfigure}[b]{0.49\textwidth}
		\centering
		\includegraphics[width=\textwidth]{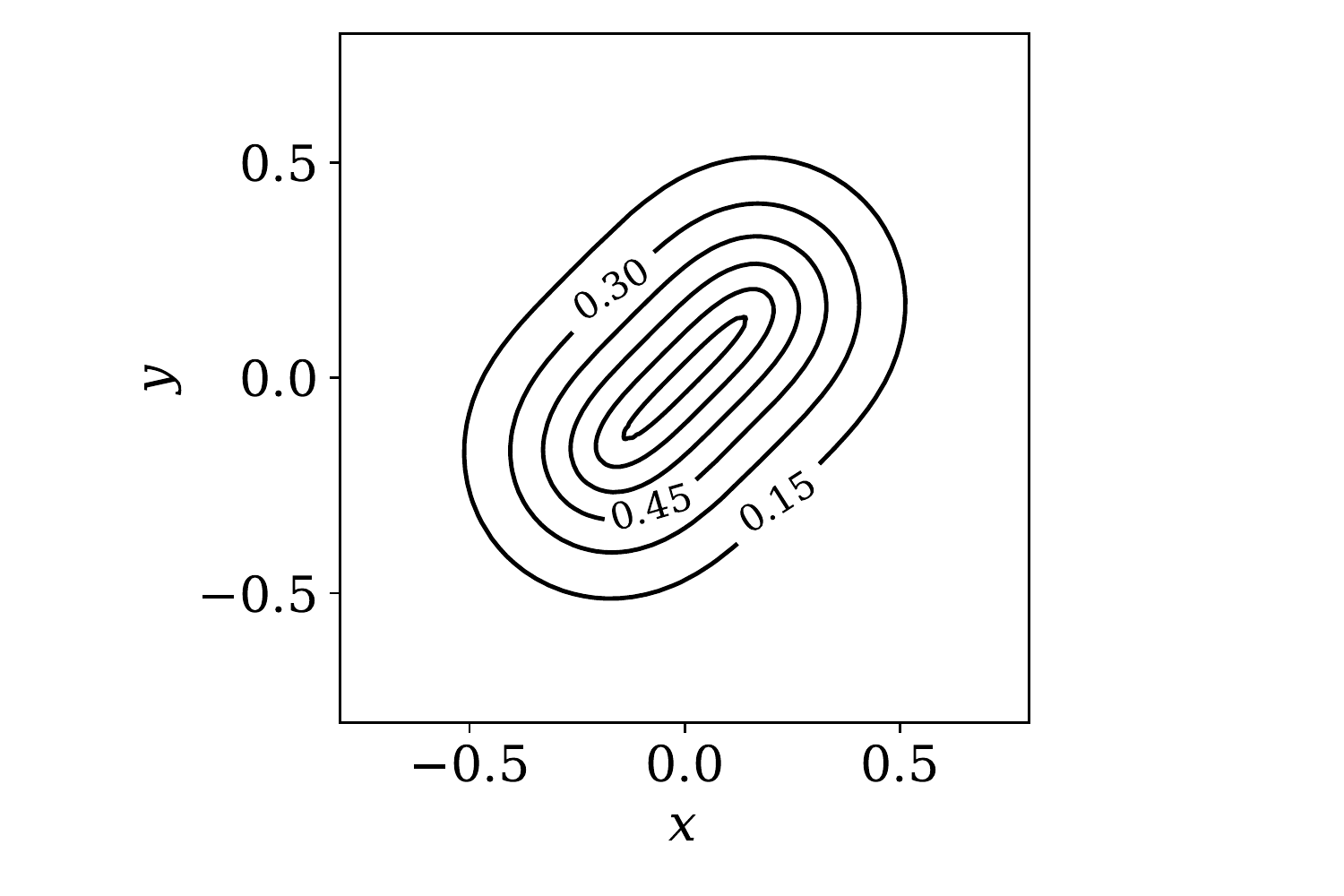}
		\caption{}
		\label{fig:Gaussians2standard}
	\end{subfigure}
	\caption{Figure \ref{fig:Gaussians4standard} shows some level sets of the eikonal depth for a Gaussian mixture model with means that are four standard deviations apart. The contribution of any of the components is negligible at two standard deviations from the other mean. Figure \ref{fig:Gaussians2standard} shows some level sets for the eikonal depth for a Gaussian mixture model with means that are two standard derivations apart. 
		\label{fig:separatedModes}
	}
\end{figure}

\begin{example}
	\label{ex:Gaussianmixture}
 Consider a Gaussian probability distribution in $\mathbb{R}^d$ with mean located at the origin and with covariance matrix $\sigma^2 I$, where $I$ is the identity matrix and $\sigma$ is a specified standard deviation. The probability density for this distribution is given by $\rho(x_1, \ldots , x_d)= \frac{1}{\sigma^d (2 \pi)^{d/2}} e^{-\frac{x_1^2+ \ldots + x_d^2}{2 \sigma^2}}$. We have that the cost of going around the mode along a circular arc of radius $r$ and angle $\theta$, which is parametrized as $\gamma(t)=(r \cos t, r \sin t)$ with $t \in [0,\theta]$ after a suitable rotation, is given by $\int_{0}^{\theta} \rho(\gamma(t)) |\dot \gamma (t)| d t =\frac{1}{\sigma^d (2 \pi)^{d/2}}\int_{0}^{\theta} e^{-\frac{r^2}{2 \sigma^2}} r d t = \frac{1}{\sigma^d (2 \pi)^{d/2}} r e^{-\frac{r^2}{2 \sigma^2}} \theta$. On the other hand, the cost of going from any point on the circular path to the mode can be calculated, after a suitable rotation, as the integral of $\rho$ along the path $\tau(t)=(t, \ldots, 0)$ with $t \in [0,r]$: $\int_{0}^{r} \rho(\tau(t)) |\dot \tau(t)| dt = \frac{1}{\sigma^d (2 \pi)^{d/2}} \int_{0}^{r}  e^{-\frac{s^2}{2 \sigma^2}} dt = \frac{1}{\sigma^{d-1} (2 \pi)^{d/2}} \sqrt{\frac{\pi}{2}} \mbox{erf} \left(\frac{r}{\sqrt{2} \sigma}\right)$.


 
 {\color{\changed} In particular, for $d=2$, the two costs are $\frac{1}{2 \pi \sigma^2} r e^{-\frac{r^2}{2 \sigma^2}} \theta$ and $\frac{1}{2 \pi \sigma} \sqrt{\frac{\pi}{2}} \mbox{erf} \left(\frac{r}{\sqrt{2} \sigma}\right)$. These two quantities are comparable if $r \sim 2 \sigma$, as can be checked numerically. So we expect that, heuristically speaking, once we are two standard deviations away from the mean it is cheaper to go around than to go over the mode.}
	
	Now consider a Gaussian mixture model in $\mathbb{R}^2$. If each component of the mixture model has covariance $\sigma^2 I$, where $I$ is the identity matrix and $\sigma$ is a specified standard deviation, then the distribution will have well-separated modes if the contribution of one mixture component is negligible approximately two standard deviations away from the mean of any other mixture component. Thus we expect, again heuristically, that if components are more than four standard deviations apart then each mode should be identified \textcolor{\changed}{as a local "center point"} by the eikonal depth.
	
	Figure \ref{fig:separatedModes} shows a computational example, where we compute the eikonal depth \textcolor{\changed}{for two mixtures of Gaussian probability distributions with $\sigma=0.25$ in which the Gaussian distributions are}, respectively, four standard deviations apart and two standard deviations apart. In the former case the modes are well-separated, according to Definition \ref{def:well-separated}, and hence the depth has two local maxima. In the latter case the Gaussians are close enough together that the modes are not well-separated, and hence we only have one local maximum of the depth.
\end{example}

\textcolor{\changed}{We now address stability and uniqueness of the eikonal depth. It is desirable for a statistical depth to be uniquely associated to the considered set of probability distributions. In this regard, it is worth pointing out that the Tukey depth completely characterizes the input distribution for finite discrete measures \cite{koshevoy2002tukey,struyf1999halfspace} and among the class of rapidly decaying distributions (i.e. those which decay faster than exponential) \cite{koshevoy2003lift}. In some settings the Tukey depth is known not to characterize the underlying distribution \cite{nagy2021halfspace}. In contrast, the following results clearly quantify in what sense the eikonal depth and their underlying distributions form a one-to-one correspondence.}

\begin{proposition}
\label{prop:eikonalcontinuity}
	\textcolor{\changed}{The eikonal depth is continuous in the input distribution. If $\rho_1,\rho_2$ are supported on the closure of a bounded and open set $\Omega$} and satisfy bounds $0<c \leq \rho_1, \rho_2 \leq C$ then
	\[
	\|D_{eik}(x,\rho_1) - D_{eik}(x,\rho_2)\|_\infty \leq  \ell \|\rho_1 - \rho_2\|_\infty.
	\]
	where $\displaystyle \ell=\frac{\mbox{diam} (\Omega) \, C}{2 c}$.
\end{proposition}
\begin{proof}
{\color{\changed}
We start by finding a bound on the maximum distance that a particle can travel starting at a point in $\Omega$. Let us denote by $v_+$ and $v_-$ denote respectively the maximum and minimum velocity allowed by the distributions $\rho_1,\rho_2$. From the assumptions we have that $v_- \geq \frac{1}{C}$ and $v_+ \leq \frac{1}{c}$. The longest possible time a particle can travel before reaching $\partial \Omega$ is then $\frac{\mbox{diam}(\Omega)/2}{v_-} \leq C \frac{\mbox{diam}(\Omega)}{2}$. If we denote by $\ell$ the longest possible distance a particle can travel from a point in $\Omega$ then we have that $\ell \leq v_+ \cdot C \frac{\mbox{diam} (\Omega)}{2} = \frac{\mbox{diam}(\Omega) C}{2c}$.
	}
	
{Let $\gamma^*_2$ be an optimal trajectory associated to the distribution with density $\rho_2$, that is, $\gamma^*_2$ realizes the infimum in Definition \ref{def:eikonalpaths} and  $\gamma^*_2(0)=x, \gamma^*_2(b) \in \partial \Omega$ for some $b \in \mathbb{R}_+$. We have that
	\begin{align}
		D_{eik}(x,\rho_1) & = \inf_{\gamma \in \mathcal{U}_x}\int_{\gamma} \rho_1 ds = \int_0^b \rho_1(\gamma^*_2(t)) |\dot{\gamma}^*_2(t)|dt \\ & =  \int_0^b \left[ \rho_2 (\gamma^*_2(t)) + (\rho_1-\rho_2)(\gamma^*_2(t))\right] |\dot{\gamma}^*_2(t)| dt  \leq D_{eik}(x,\rho_2) +\ell  \| \rho_1 - \rho_2 \|_{\infty}
	\end{align}
}
\end{proof}

\begin{proposition}
   Assuming that $\phi$ is injective, then any probability distribution with a continuous probability density $\rho$ with support on an open and bounded domain $\Omega \subset \mathbb{R}^d$ or on $\Omega=\mathbb{R}^d$ 
  is uniquely determined by its $\phi$-eikonal depth. 
	\label{prop:depth-char-dist}
\end{proposition}

\begin{proof}
  \textcolor{\changed}{Assume that we have to distinct probability distributions $F_1,F_2$ with continuous densities $\rho_1,\rho_2$ such that $D_{eik}(x,F_1)=D_{eik}(x,F_2)$.} Any viscosity solution to the eikonal equation will satisfy the equation classically at almost every point \textcolor{\changed}{in its domain} \cite{bardi2008optimal}. {\color{\changed} Hence  $D_{eik}(x,F_1)$ and $D_{eik}(x,F_2)$ will have a common point of differentiability at a point where $\rho_1(x) \neq \rho_2(x)$} \textcolor{\changed}{By the injectivity of $\phi$, this implies that the the gradients of $D_{eik}(x,F_1)$ and $D_{eik}(x,F_2)$  will not match at $x$ and therefore the depths cannot be the same}.
\end{proof}

\textcolor{\changed}{Propositions \ref{prop:eikonalcontinuity} and \ref{prop:depth-char-dist}} have established a type of robustness of the eikonal depth. The following proposition establishes a scenario where the eikonal depth is not robust.

\begin{proposition}\label{prop:no-breakdown}
	Let $\rho$ be a continuous density function, and suppose that $\phi(s) = s^\alpha$  with $\alpha \geq 1/d$. Then by modifying an arbitrarily small amount of density we can modify the eikonal depth at a single point by \textcolor{\changed}{any desired} amount.
	\label{prop:upper-stability}
\end{proposition}

\begin{proof}
	We can decrease the depth at the point $x$ by carving a very thin path to that point. For $\alpha > 1/d$ we can increase the depth at a point arbitrarily by adding mass in a small enough ball centered at $x$. For $\alpha=1/d$ we truncate the profile $|x|^{-d}$, as in the second part of Proposition \ref{prop:median-UB}.
\end{proof}

While the previous proposition does establish a type of non-robustness, we notice that the types of modifications we made in the proof would only modify the depth in small regions. It seems unlikely that small modifications to the distribution would be able to modify the depth in large regions, but we do not pursue proving such a property here.

\textcolor{\changed}{To wrap up this section, we also give two conceptual properties, related to the optimal control formulation of the eikonal depth, which allow us to bound and to simplify computation of the same. These properties do not have do not have immediate analogs for other depth functions, but we believe they are useful and worth mentioning. First, in many settings it is possible to give upper or lower bounds on probability densities, and it seems natural to try to use those to provide bounds on the associated depth functions. The following proposition gives a comparison principle for the eikonal depths that relates upper and lower bounds for the eikonal depth to upper and lower bounds for the probability distribution. This proposition cannot be directly applied to two probability densities, as one cannot have a global inequality on such densities. However, using the scaling invariance property from Proposition \ref{prop:scale-invariant}, in some settings one can use the following proposition to provide upper and lower bounds on eikonal depths.}

\begin{proposition}
	The eikonal depth satisfies a comparison principle in the sense that if $\rho_1 (x) \leq \rho_2 (x)$ for all $x \in \mathbb{R}^d$ then $D_{eik}(x,\rho_1) \leq D_{eik}(x,\rho_2)$ for all $x \in \mathbb{R}^d$. 
	\label{prop:comparison}
\end{proposition}

\begin{proof}
	This is immediate from the control formulation of the eikonal depth, in the sense that if $\rho_1 \leq \rho_2$ then the cost associated with any $u \in \mathcal{U}_x$ will be greater when replacing $\rho_1$ with $\rho_2$ in \eqref{eq:energyEikonal}, and hence the infimums will satisfy the same inequality.
\end{proof}

\textcolor{\changed}{Next we give a basic property of the eikonal depth, which is fundamental for numerical approximation schemes. Heuristically, this property implies that the eikonal depth can be determined using the value of the density at $x$ and the value of the depth at neighboring points, which permits us to construct efficient, local, numerical schemes.}

	\begin{proposition}
	\label{prop:templabel}
		Let $U\subset \R^d$ and suppose that we know the values of $D_{eik}$ on the boundary of $U$. Then the values of $D_{eik}$ can be determined using only the density of $\rho$ inside $U$.
		\label{prop:DPP}
	\end{proposition}
	The proof of this proposition is immediate from the dynamic programming principle: any optimal path beginning at a point in the interior of $U$ will also be optimal from the time it leaves $U$ onward. \textcolor{\changed}{ Indeed, we already used this idea in the proof of Proposition \ref{prop:well-separated}. This simple observation also factors prominently in the construction of numerical methods, see Section \ref{sec:numerics}.
}


	\subsection{Isometric robustness}
\label{sec:robustness}
	\begin{proposition}
	  Let $\Phi:X \to X$ be an invertible, $C^1$ mapping so that
	  \[
	   \textcolor{\changed}{ \|D\Phi-I\|_{\infty} \leq \e.}
	  \]
	  Let $\tilde \rho = \Phi_\sharp \rho$\textcolor{\changed}{, that is $\int_B \tilde{\rho}(y) dy= \int_{\Phi^{-1}(B)}\rho (x) dx$}. Then 
	  \begin{equation}
(1-\e) D_{eik}(x,\rho) \leq D_{eik}(\Phi(x),\tilde \rho) \leq (1+\e) D_{eik}(x,\rho), \mbox{ for all } x \in \mathbb{R}^d
	    \label{eqn:iso-robust}
	  \end{equation}\label{prop:iso-robust}
	\end{proposition}

	\begin{proof}
{ 
	  Consider a path $\gamma_1(t)$ satisfying $\gamma_1(0) = x$ and $\gamma \to \infty$, and define a second path $\gamma_2(t) = \Phi(\gamma_1(t))$. We may readily compute 
	  \begin{align}
	    &\int_{0}^{\infty} \phi( \tilde \rho(\gamma_2(t)) )|\dot \gamma_2(t)| \,dt \\
	    &= \int_{0}^{\infty} \phi( \tilde \rho(\Phi(\gamma_1(t)) ) |D\Phi (\gamma_1(t)) \dot \gamma_1(t)| \,dt \\
	      &= \int_{0}^{\infty} \phi( \rho(\gamma_1(t)) ) |D\Phi (\gamma_1(t)) \dot \gamma_1(t)| \,dt,
	    \label{eqn:COV-robust}
	  \end{align}
	  where the last line follows because of the definition of $\gamma_2$ and $\tilde{\rho}$. By using the assumption on $|D\Phi|$, we immediately have that $(1-\e)|\dot \gamma_1| \leq |D\Phi (\gamma_1(s)) \dot \gamma_1(s)| \leq (1+\e) |\dot \gamma_1|$, which readily implies that
	  \begin{displaymath}
	    (1-\e) \int_{0}^{\infty} \phi(\rho(\gamma_1(s)) )|\dot \gamma_1(s)| \,ds \leq \int_{0}^{\infty} \phi( \tilde \rho(\gamma_2(t)) )|\dot \gamma_2(t)| \,dt \leq (1+\e) \int_{0}^{\infty} \phi( \rho(\gamma_1(s)) )|\dot \gamma_1(s)| \,ds.
	  \end{displaymath}
	  By taking the infimum over paths we then obtain that, \textcolor{\changed}{for all $x \in \mathbb{R}^d$,}
	  \[
	    (1-\e) D_{eik}(x,\rho) \leq D_{eik}(\Phi(x),\tilde \rho) \leq (1+\e) D_{eik}(x,\rho),
	  \]
	  as desired.
}
	\end{proof}

	The fact that the eikonal depth is isometrically robust is a consequence of the \textcolor{\changed}{inherent metric nature of its definition}. Our viewpoint is that the eikonal depth is a measure of centrality or outliers in metric geometry in the same sense that the Tukey depth is a notion of centrality for convex geometry. Indeed, the very definitions, which are, respectively, based upon path length versus supporting halfspaces, are direct consequences of the underlying geometric viewpoint.
	
	The notion of isometric robustness that we propose is quite different from the notion of breakdown point used in classical robust statistics. However, it is, in spirit, much closer to the study of distributionally robust optimization problems \cite{chen2020robust}, or more generally adversarial training problems in statistics. In those contexts, one typically solves an optimization problem of the form
	\[
	\min_{\theta \in \Theta} \sup_{d(\mu,\tilde \mu) \leq \e} \mathbb{E}_{\tilde \mu} R(\theta).
	\] 
	Here $\mu$ is an underlying data distribution, for example associated with a classification or regression problem, and $R(\theta)$ represents a risk function which depends upon the measure of the data points. The metric $d$ represents the geometry of permissible data perturbations by a hypothetical adversary, and $\e$, called the adversarial budget, represents the power that adversary has to corrupt the data. In many common adversarial learning problems the metric $d$ is some type of Wasserstein metric (see e.g. \cite{MurrayNGT}), which permits adversaries to move many data points smaller distances. These adversaries are in one sense much weaker than the adversary supposed in studying the breakdown point: they typically cannot move points very far. But on the other hand they may move many points short distances, which makes these adversaries in \textcolor{\changed}{another} sense stronger than the one associated with the breakdown point. Utilizing this type of adversary has been incredibly effective in improving generalization for statistical algorithms in the context of deep learning \cite{goodfellow2015}, and has sparked significant algorithmic and theoretical work \cite{bungert2021geometry,MurrayNGT,madry2019deep,moosavi2019robustness,pydi2019adversarial}.	

	
	One can view the isometric robustness in Proposition \ref{prop:iso-robust} as following in the spirit of DRO problems, in that we are imagining an adversary that is given a limited adversarial budget expressed in terms of how much they can skew distances between points. Such adversaries are neither strictly weaker or stronger than the one imagined in studying the breakdown point, but are probably closer in practice to the robustness studied in contemporary DRO and adversarial learning problems \textcolor{\changed}{based upon Wasserstein distances}. \textcolor{\changed}{It is natural to consider whether classical depths, such as the Tukey depth, are robust in the sense of Proposition \ref{prop:iso-robust}. The following example shows that this is not the case.}

	\begin{example}
	  Suppose that $\rho$ is given by a uniform distribution on two balls in $\R^2$ of radius $1$ at the points $(0,\pm K)$. At the origin it is straightforward to compute that the Tukey depth will be $1/2$. However, consider a mapping $\Phi$ which leaves the two balls invariant, but moves the origin to $(1,0)$. Such a mapping can be constructed so that $|D\Phi - I| \leq \frac{3}{K}$. However, the Tukey depth of the new point will be $0$. This demonstrates that the Tukey depth is not necessarily isometrically robust.
	\end{example}

	The previous example clearly shows that the Tukey depth is \emph{not} robust with respect to approximate isometries. This should not be surprising, as the Tukey depth is really built upon definitions from convex geometry, \textcolor{\changed}{as opposed to metric geometry}.

{\color{\changed}

  \subsection{Summary of properties}

  Thus far we have proved quite a few different properties of the eikonal depth. Along the way we have also provided discussion and comparison with the halfspace/Tukey depth. For convenience, here we provide a brief summary of these properties.

  \begin{itemize}
    \item \textbf{Scaling:} The eikonal depth is invariant under rigid transformations (Proposition \ref{prop:rigid}), and is weakly scale invariant when $\phi(s) = s^\alpha$ (Proposition \ref{prop:scale-invariant}). It is not invariant under general linear transformations. 
    \item \textbf{Bounds:} The eikonal depth does not admit direct upper or lower bounds of their maxima (Proposition \ref{prop:median-UB} and Example \ref{ex:median-LB}). The example for large maximal values required densities which are unbounded from above. The example for small maximal values required that the density have large gradient. It does admit a convenient comparison principle, which in some cases can aid in bounding the depth (Proposition \ref{prop:comparison}).
    \item \textbf{Multi-modality:} The eikonal depth has local maxima near prominent modes of the density (Proposition \ref{prop:well-separated}. Its level sets are not necessarily convex.
      \item \textbf{Robustness:} The eikonal depth uniquely determines the underlying density (Proposition \ref{prop:depth-char-dist}). It admits a quantifiable stability in terms of the underlying density (Proposition \ref{prop:eikonalcontinuity}. It also  possess a robustness to nearly isometric distortions (Proposition \ref{prop:iso-robust}), which other classical depths do not enjoy. The eikonal depth is not robust in the sense of breakdown point (Proposition \ref{prop:no-breakdown}).
	\item \textbf{One-dimensional equivalence:} The eikonal depth is equivalent to the quantile depth in one-dimension (Section \ref{sec:onedimensional}).
  \end{itemize}

\subsection{What properties should a statistical depth have?}
\label{sec:otherdefs}

In the one-dimensional setting, there are natural ways to define center-outward orderings of data, as described via medians and quantile depths in Section \ref{sec:onedimensional}. However, in higher dimensions there is no single canonical center-outward ordering associated with data points. This has led to myriad definitions of depths in higher dimension, as discussed in Section \ref{sec:literature}. This multitude of definitions naturally leads to the question of what it even means to be a statistical depth.

Several classical works have attempted to address this question. 
A set of desirable properties for a statistical depth satisfied by many commonly used statistical depths, such as the Mahalanobis depth, the Tukey depth, and the simplicial depth is given by Liu-Zuo-Serfling \cite{zuo2000general}. Serfling and Zuo consider that the defining properties of statistical depth are affine invariance, maximality at center, monotonicity with respect to central point, and vanishing at infinity. In their definition there exists a unique central point with maximal depth, the center, and the depth is decreasing along any ray emanating from that point. These characteristics define a center-outward ordering using the nested level sets of the depth function. Indeed, some works treat their definition as the de-facto definition of what it means to be a statistical depth.
	
As Serfling and Zuo point out \cite{zuo2000general}, when defining a statistical depth one needs to make a choice between having a center-outward ordering and the ability to capture multimodality. This is a direct consequence of the monotonicity property, which precludes the existence of several central points. Liu proposed a model of depth that prioritizes multimodality in \cite{liu1999multivariate}. Freiman and Meloche \cite{fraiman1999multivariate} proposed a depth in this same direction, the likelihood depth, that is based on estimations of multivariate kernel densities. More recently, Lok and Lee \cite{lok2011new} have proposed a depth focused on multimodality that is based on quantiles of interpoint distances.
	
The ability to deal with multimodality in the notion of a depth is a necessity for some applications: for instance if the data is well-clustered and the ordering associated with a depth is meant to preserve those clusters \cite{liu1999multivariate}. Even more generally, most classical depth functions produce nested orderings of convex sets, which may not be desirable for distributions with non-convex support (see e.g. the discussion in \cite{chernozhukov2017monge}).

The depth that we have constructed has properties which are distinct from many classical depth functions. Indeed, it fails to satisfy many of the points in the definition given in \cite{zuo2000general}. However, it flexibly captures multi-modal and clustered data, still satisfies natural notions of stability and robustness, is easily computable, and is easy to interpret. Even in the Euclidean setting, we posit that this depth provides a competitive alternative to many of the classical notions of depth. We shall also see, in subsequent sections, that this depth extends in very natural ways to non-Euclidean settings, further highlighting its versatility and potential applicability.

	
}

\section{Extension to non-Euclidean settings}
\label{sec:extensionNonEucli}

In settings where data is non-Euclidean, the question of how to appropriately define depths and medians becomes even less clear. Although there have been some attempts to generalize the halfspace depth to this setting (see e.g. \cite{carrizosa1996characterization,small1997multidimensional}), in general it is not clear how to best approach this problem. We shall see in this section that the definition of the eikonal depth extends very naturally to a wide class of metric settings, while still remaining interpretable and computable.

To begin, we recall our original definition of the eikonal depth:

\[
D_{eik}(x,F):=\inf_{\gamma \in \mathcal{U}_x} \int_{\gamma} \rho ds =\inf_{\gamma \in \mathcal{U}_x} \int_{0}^{\infty}\textcolor{\moreedits}{ \gamma (t) |\dot {\gamma} (t)|} dt,
\]

This definition, as it turns out, only requires a few basic ingredients, which are present in most metric settings. In particular, this definition extends very naturally as long we can have a way of defining the following:
\begin{enumerate}[(1)]
  \item A notion continuous curves, and speed of travel along those curves
  \item A consistent notion of probability density along those curves
	\item An appropriate boundary condition, namely that our curves approach ``spatial infinity'' as $t \to \infty$.
\end{enumerate}

%
%
%
%

Each of these requirements are immediately fulfilled in many non-Euclidean settings. In particular, we now give informal examples of different settings where this definition may be directly extended. In doing so we focus on conceptual ideas, at the cost of perhaps some precision and rigor.

\begin{example}[Manifolds with boundary]
  Suppose that $\mathcal{M}$ is a $k$-dimensional Riemannian manifold with non-empty boundary. On such a manifold, by using our Riemannian metric we can define \textcolor{black}{a notion of curves, namely a mapping $\gamma:[0,\infty) \to \mathcal{M}$ which is differentiable so that $\dot \gamma(t) \in T\mathcal{M}$. We then suppose that we have a continuous density function $\rho: \mathcal{M} \to [0,\infty)$. Finally, we let the boundary of our manifold define a the ``boundary condition'', namely we require that $\gamma(t) \to \partial \mathcal{M}$. Let us define $\mathcal{U}_x$ to be the set of such curves which satisfy this boundary condition, and begin at $x \in \mathcal{M}$ at time $0$. Under such assumptions, we can define a depth function using
	\[
	D_{eik}(x,\rho):=\inf_{\gamma \in \mathcal{U}_x} \int_{0}^{\infty} \rho (\gamma(t)) |\dot \gamma(t)| dt 
	\] }
	We notice that such a definition is actually identical to the Euclidean one, except that we have had to modify the class $\mathcal{U}_x$, \textcolor{\moreedits}{to encode the geometry of the manifold through the boundary condition}. We also notice that this framework could be translated into computing the distance of the point $x$ to the boundary $\partial \mathcal{M}$ under a new metric $\tilde g = \rho^{k/2} g$, where $g$ is the original metric for $\mathcal{M}$.

\end{example}

\begin{example}[Unions of Riemannian manifolds]
  We consider a finite collection of Riemannian manifolds $\mathcal{M}_i$ with boundary, of possibly unequal dimension. We suppose that these manifolds admit some points of intersection: a canonical example would be when these manifolds are smooth surfaces embedded in Euclidean space, with non-trivial intersection. We then call a \textcolor{\moreedits}{mapping $x:[0,\infty) \to \cup_{i=1}^p \mathcal{M}_i$ a curve} if there exists a partition of $[0,\infty)$ into closed intervals $[a_j,a_{j+1}]$, and so that for each sub-interval $x$ \textcolor{\moreedits}{is a curve of on one of the manifolds $\mathcal{M}_i$}. Again, we can define the boundary condition using the boundary of the manifolds, in this case we let it be $\cup_{i=1}^p \partial \mathcal{M}_i \backslash \cup_{i=1}^p \mathcal{M}_i$. We then use exactly the same definition for the eikonal depth.
\end{example}

\begin{example}[Geodesic metric spaces]
  {\color{\changed}
  Let us consider a geodesic metric space, which heuristically allows us to define distances between points using the length of continuous paths connecting those points. More specifically, we may define the length of a curve $\gamma:\R \to X$ between two times $t_1,t_2$ via
    \[
      \ell(\gamma,t_1,t_2) := \sup_{(s_i)_{i=1}^{k+1} \in \mathcal{T}_{[t_1,t_2]}} \sum_{i=1}^k d(\gamma(s_i),\gamma(s_{i+1})).
  \]
 Here $\mathcal{T}_{[a,b]}$ is the set of all finite partitions of $[a,b]$, and the $t_i$ are the boundaries of those partitions.
 A geodesic metric space is simply a space so that, if $Q_{xy}$ is the set of all continuous paths starting at $x \in X$ at time $0$ and ending at $y \in Y$ at time $1$, then $\min_{\gamma \in Q_{xy}} \ell(\gamma,0,1) = d(x,y)$.

    Geodesic metric spaces generalize Riemannian manifolds in that they do not require local Euclidean structure for length. Classical examples include Banach, sub-Riemannian, and Finsler spaces. One introductory text on these types of spaces is \cite{burago2001course}. This notion, which is very general, does require the existence of continuous curves connecting points, and hence precludes discrete metric structure: we address this direction in our next example.

    We choose to simply let $\mathcal{O} \subset X$ denote our ``boundary'': other settings could also be considered.  {We define $\mathcal{U}_x$ to be the set of continuous functions  $\gamma:[0,T] \to X$ such that $\lim_{t \to T} d(\gamma(t),\mathcal{O}) = 0$. 

  Given such a metric space, and an associated probability measure on $X$, associated with a continuous density $\rho$, we then define
  \[
    D_{eik}(x,\rho) = \inf_{\mathcal{U}_x} \inf_{(t_i)_{i=1}^{K+1} \in \mathcal{T}_{[0,T]}} \sum_{i=1}^K \max_{s \in [t_i,t_{i+1}]} \rho(\gamma(s)) \ell(\gamma,t_i,t_{i+1}).
  \] }
  This essentially corresponds to an upper Riemann approximation of the integral used previously. Such a definition requires very little on the underlying spaces, and provides an extremely general framework in which we can sensibly extend the definition of the eikonal depth.
}

\end{example}

\begin{example}[Graphs]
  Consider a weighted graph $(X,E,W)$, where $\{x_i\}_{i=1}^N$ are the vertices, $E \subset \{1,\dots,N\}^2$ and $W:E \to \R^+$. If the graph is connected, the natural definition of a metric on the graph is given by
  \[
    d(x_1,x_2) = \inf_{ p \in \mathcal{P}_{x_1,x_2}} \sum_j w_{p(j),p(j+1)}, 
  \]
  where $\mathcal{P}_{x_1,x_2}$ are the set of admissible paths connecting $x_1,x_2$, namely finite lists of elements $(p(j))_{j=1}^k$ so that $(p(j),p(j+1)) \in E$ for all $j$, and so that $p(1) = x_1$ and $p(k) = x_2$ (abusing notation here slightly).

  We now suppose that each vertex has an associated density, expressed by a function $\rho: X \to \R^+$. As the path length is edge weighted and the density is vertex weighted, one has to make a choice in defining density-weighted path lengths. One such choice could be to define
  \[
    d_\rho(x_1,x_2) = \inf_{p \in \mathcal{P}_{x_1,x_2}} \sum_j w_{p(j),p(j+1)}\frac{(\rho(p(j)) + \rho(p(j+1)))}{2}.
  \]
  With this choice any path which passes through $x_i$ will be charged $\rho(x_i)$ times the average of the weights of the incoming and outgoing edges, which serve as analogs for the path velocities. As in the case of Riemannian metrics, we can readily see that one can view this formula as simply solving the shortest path problem with new edge weights.

  Suppose that we identify some number of vertices $\mathcal{O}$ as the ``boundary'' of our graph. Then we can readily define
  \[
    D_{eik}(x) = \inf_{y \in \mathcal{O}} d_{\rho}(x,y) =  \inf_{y \in \mathcal{O}}\inf_{p \in \mathcal{P}_{x,y}} \sum_j w_{p(j),p(j+1)}\frac{(\rho(p(j)) + \rho(p(j+1)))}{2}.
  \]


%
\end{example}

These examples make it clear that this framework is extremely flexible. Indeed, the final framework even applies directly to discrete data clouds which are embedded as graphs, of which we give several examples in Figure \ref{fig:manifolds}. The computational methods used for solving eikonal equations in Euclidean settings have immediate counterparts in the graph settings, which we describe in Section \ref{sec:graph_setting}. As long as our graph and densities are constructed using an asymptotically consistent approximation, then the graph setting actually provides a very convenient way to approximate depths using empirical samples coming from Euclidean or manifold settings. \textcolor{\changed}{Some care must be taken in constructing these graphs, but recent work (see, for example, \cite{trillos2016continuum,trillos2018error} and the references therein) provides many consistent choices. The implementation is straightforward and indeed} this is actually how we approximated our depth functions in Figures \ref{fig:manifolds} and \ref{fig:MNIST} without needing to construct meshes.


\textcolor{\changed}{One significant complication is the need to specify} a ``boundary''. In the Euclidean setting, the boundary at infinity is very natural. Similarly, when one considers a manifold with boundary then the same is explicit. On a general metric space or graph it is not obvious how to select an appropriate boundary set. In the context of manifold learning\textcolor{\changed}{, i.e. when discrete samples are drawn from a manifold with boundary,} there are certain established methods for identifying boundary points on empirical measures. However, even with these inherent challenges, the generality of this depth to various settings can be seen as a significant strength.

\section{Computational approaches}

In the previous sections, we have identified the solutions to the optimization problem \eqref{eq:energyEikonal} as the viscosity solutions of the partial differential equation \eqref{eq:eikonal}. In turn, we identified a number of desirable properties which are enjoyed by those solutions. However, we have not, up to this stage, addressed, in any way, how one can approximate solutions to these equations: indeed as the equation is non-linear it is genereally not possible to find closed form solutions. Furthermore, in the previous sections we have always assumed that we have knowledge of $\rho$, whereas in practice one often only has access to empirical quantities. Fortunately, various numerical methods have been previously developed for solving eikonal-type equations both on grids and on more general graphs, which allow us to efficiently approximate the eikonal depth in a variety of settings. We do not attempt to give a complete accounting of these numerical methods, instead opting to describe high level ideas and algorithms, along with examples where we have implemented these schemes.

\subsection{Numerical methods, monotone schemes, and fast marching methods}
\label{sec:numerics}

In constructing a numerical method for approximating solutions to \eqref{eq:eikonal}, care needs to be taken that the scheme will correctly select the (unique) viscosity solution of the first-order partial differential equation, as opposed to one of the (infinitely many) weak solutions to the equation. For concreteness, let us consider the following example:

\begin{example}
	Consider $\rho(x) \equiv 1$ on the domain $[0,1]$, namely a uniform distribution. The correct viscosity solution to \eqref{eq:eikonal} will be given by $u(x) = \min(x,1-x)$. Suppose that our goal is to approximate $u$ on the points $x_i = i/N$, namely to solve for values of $u_i \approx u(x_i)$ so that a discrete approximation of \eqref{eq:eikonal} is satisfied. One natural numerical scheme would be to use a finite difference scheme based upon Taylor approximations, namely $N \cdot |u_i - u_{i-1}| =  \rho(x_i) = 1$. However, such a scheme fails to necessarily select the correct solution: for example letting $u_i = 1/N$ for odd $i$ and $0$ for even $i$ (when $N$ is even) would solve this discrete equation but in no way resembles the viscosity solution.
\end{example}

However, there are a variety of methods for approximating derivatives \cite{barles1991convergence} which are known to be consistent with viscosity solutions. In particular, one can use \emph{monotone methods} to construct such consistent derivative approximations. For example, in the notation of the previous example, in one dimension one can approximate
\begin{equation}\label{eqn:monotone-scheme}
u'(x_i) \approx D_m u (x_i) :=  \begin{cases} N \cdot |u_i - u_{i-1}| &\text{ if } u_{i-1} \leq u_i, \\ N \cdot |u_{i+1}-u_i| & \text{ else if } u_{i+1} \leq u_i, \\ 0 &\text{ otherwise.} \end{cases}
\end{equation}
This finite difference scheme selects a ``direction'' for approximating the derivative based upon the slope of $u$ itself. \textcolor{\changed}{In particular, $D_m$ only uses neighboring points to compute the derivative if they have smaller value tha $u_i$.} Such a scheme is called monotone because it preserves a type of comparison principle: in the context of the eikonal equation this means that if $u \leq v$ and if $u(x_i) = v(x_i)$ then $|D_m u(x_i)| \geq |D_m v(x_i)|$. This type of comparison result is known to imply that these approximation schemes are consistent with viscosity solutions. \textcolor{\changed}{The approximation \ref{eqn:monotone-scheme} admits a direct extension to grids in higher dimension.}

\textcolor{\changed}{
However, even after selecting a consistent derivative approximation, one still needs to solve a system of non-linear equations: namely one needs $|D_m u| = \rho$. The \emph{fast marching method} \cite{sethian1999level} offers an efficient and consistent way to solve these equations. From a high level, this method tracks a set on which the solution is known (i.e. the ``solved set'', which is initialized using boundary conditions), and then solves for the values of $u$ on the set of points which neighbor the solved solved set (known as the ``considered set'') using an appropriate monotone derivative approximation scheme (e.g. in the spirit of Equation \ref{eqn:monotone-scheme}). The point in the considered set which has the smallest $u$ value is then added to the solved set. This process is repeated until all points have been added to the solved set. In essence, this scheme uses the content of Proposition \ref{prop:DPP}, which is intimately connected to the dynamic programming principle, in order to construct solutions in an efficient manner. The overall computational complexity of such a method is, up to logarithmic factors, $O(N)$, where $N$ is the total number of points in the grid or point cloud.}

\subsection{Difference scheme on point clouds}
\label{sec:graph_setting}
In many settings it may be desirable to construct an approximate solution to the eikonal equation by directly using sampled data points instead of constructing a solution on a computational grid. Such an approach avoids the need to construct higher-dimensional grids and has the potential to be more efficient and data intrinsic.

However, it is necessary to carefully consider the manner in which we approximate our derivatives in the context of using the empirical distribution as our underlying discrete object. \textcolor{\changed}{We will follow the underlying framework popularized in studying total variation and Laplacian regularization in learning problems \cite{garcia2020maximum,slepcev2019analysis}}. We do not attempt to rigorously justify the numerical method we propose here, instead opting to provide an accessible, heuristic justification.

To begin, we consider a sequence of data points $\{X_i\}_{i=1}^n \subset \R^d$ which are IID samples from a distribution with density $\rho$. We then construct a weighted graph on this set of points by defining graph weights via the relation
\begin{displaymath}
w_{ij} = \sigma n^{-1}h^{-d} \eta\left( \frac{|X_i-X_j|}{h} \right).
\end{displaymath}
Here $\eta:\R \to \R$ is a kernel describing how related we consider points based upon their distance from one another. Common examples of such a kernel include Gaussians or indicator functions. The parameter $h$ is a kernel bandwidth, which determines how we are scaling the similarities between points. From a heuristic level we consider two points to be related if they are within distance $h$ of each other. The parameter $h$ needs to be sufficiently small, but not too small: in practice one ought to select $1 \gg h \gg n^{-1/d}$.  The parameter $\sigma$ is a normalizing constant, which is chosen so that we satisfy the following integral relation
\begin{displaymath}
\frac{\sigma}{2} \int_{\R^d} \eta(|x|) x_1^2 \,dx = 1.
\end{displaymath}

We may then propose \textcolor{\changed}{using the following equation as our differencing scheme, which is based on the one introduced in \cite{desquesnes2013eikonal}}:
\begin{displaymath}
\rho_i^2 = \sum_{j \sim i} h^{-2} w_{ij} (0,u_i-u_j)_+^2.
\label{eq:differenceScheme}
\end{displaymath}

We now heuristically explain why this is a consistent approximator. If we hold $h$ fixed and then take $n \to \infty$, the previous sum, which is divided by $n$ in the weights, is well-approximated by an integral, and we obtain the relation
\begin{displaymath}
\rho(x)^2 = \int_{\R^d} \frac{\eta(|x-y|/h)}{h^{d+2}}(0,u(x)-u(y))_+^2\,dy.
\end{displaymath}
Then using a Taylor approximation and neglecting the quadratic terms, we obtain
\begin{displaymath}
\rho(x)^2 = \int_{\R^d} \frac{\eta(|x-y|/h)}{h^d} (0,\nabla u(x) \cdot (x-y) h^{-1} )_+^2 \,dy.
\end{displaymath}
Then using a change of variables $z = (x-y)/h$ we obtain
\begin{displaymath}
\rho(x)^2 = \int_{\R^d} \eta(|z|) (0,\nabla u(x)\cdot z)_+^2 = \frac{|\nabla u(x)|^2}{2} \int_{\R^d} \eta(|z|) z_1^2 \,dz = |\nabla u(x)|^2.
\end{displaymath}
The choice of $\sigma$ then immediately gives the desired equation.

Two steps in the previous derivation were informal at this stage. First, we used a Taylor approximation to convert a difference to a derivative: this step is justified in the limit $h \to 0$. The second step was where we approximated the sum with the integral, with fixed $h$. This is justified as long as the number of points in a ball of radius $h$ goes to infinity as $n \to \infty$: this motives the requirement that $h \gg n^{-1/d}$. All of these steps have been rigorously justified within the context of graph Laplacians, and we expect that similar rigorous justification would hold here. \textcolor{\changed}{Analyses of this type have been carried out in \cite{trillos2018error}, for the Laplace-Beltrami operator, and in \cite{fadili2021limits}, for a time-dependent eikonal equation.}


\subsection{Numerical illustrations}

In this section, for the sake of illustration, we give a few more involved numerical examples, along with some details. These examples are meant to demonstrate the potential of this type of depth, and its applicability in a wide range of situations. While these examples are clearly artificial, we view them as promising, and as warranting further investigation into the application of this statistical depth.

{\color{\changed} \textbf{Example 1: Gaussian mixtures in $\R^2$.} First, in Figures \ref{fig:GaussiansAlpha} and \ref{fig:GaussianMixture}, we have demonstrated level curves of the eikonal depths and its variants for different Gaussian mixtures. These examples were computed using a grid based function representation with resolution 1/128, using analytic density values and the standard fast marching method. Such methods enjoy excellent convergence properties, and in particular converge toward the analytical solution of the continuum differential equation with an error proportional to the grid spacing. These examples help to provide intuition and are easy to visualize.

\begin{figure}
	\centering
	\begin{subfigure}[b]{0.45\textwidth}
		\centering
		\includegraphics[width=\textwidth]{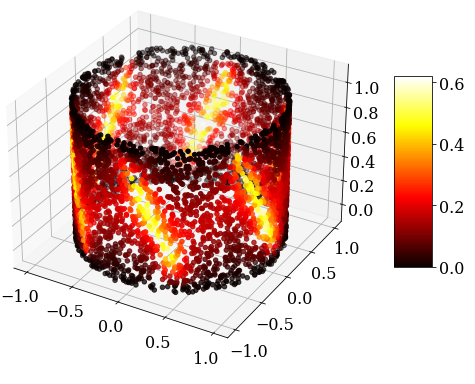}
		\caption{}
	\end{subfigure}
	
	\begin{subfigure}[b]{0.45\textwidth}
		\centering
		\includegraphics[width=\textwidth]{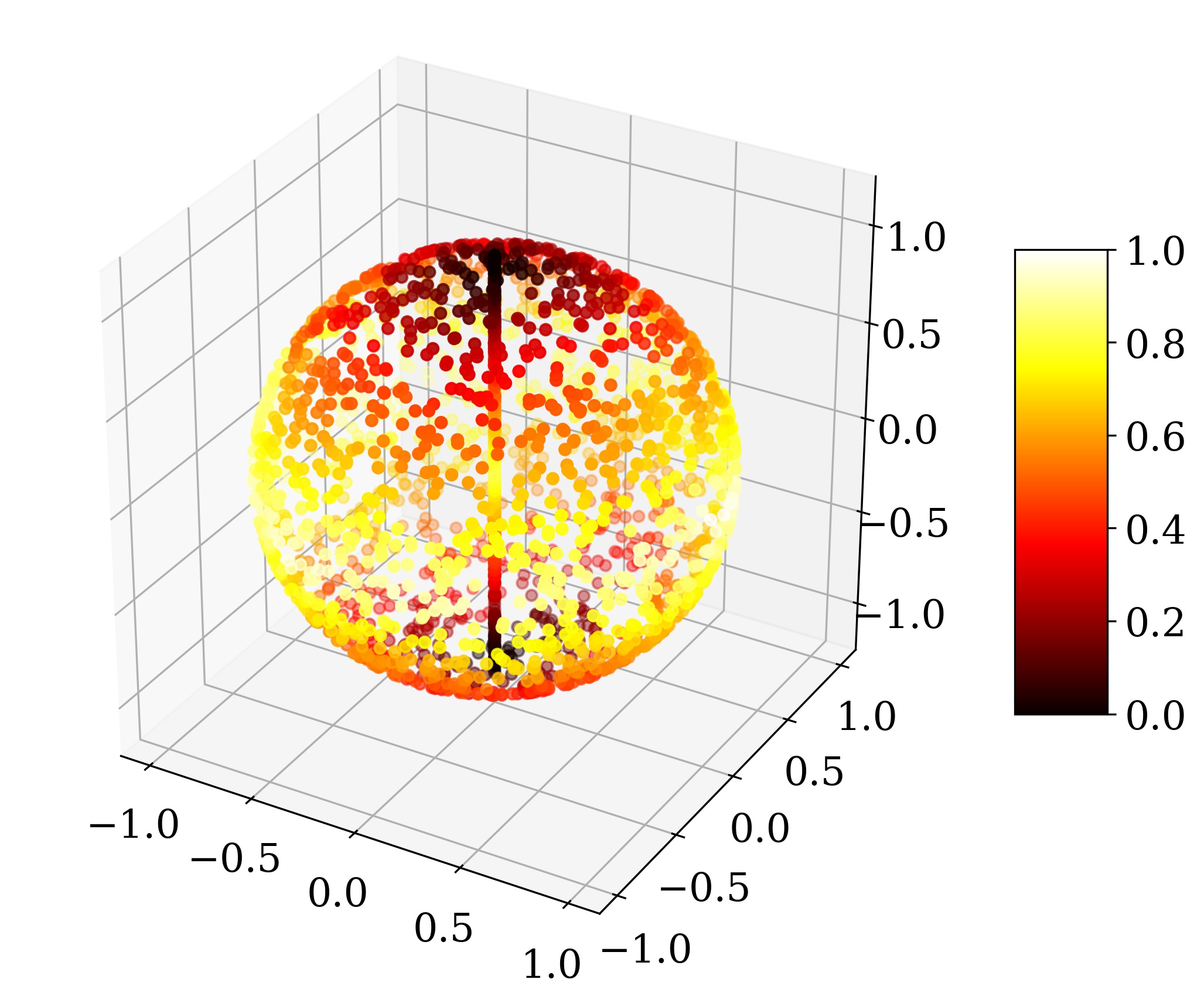}
		\caption{}
		\label{fig:depth_sphereline}
	\end{subfigure}
	\begin{subfigure}[b]{0.45\textwidth}
		\centering
		\includegraphics[width=\textwidth]{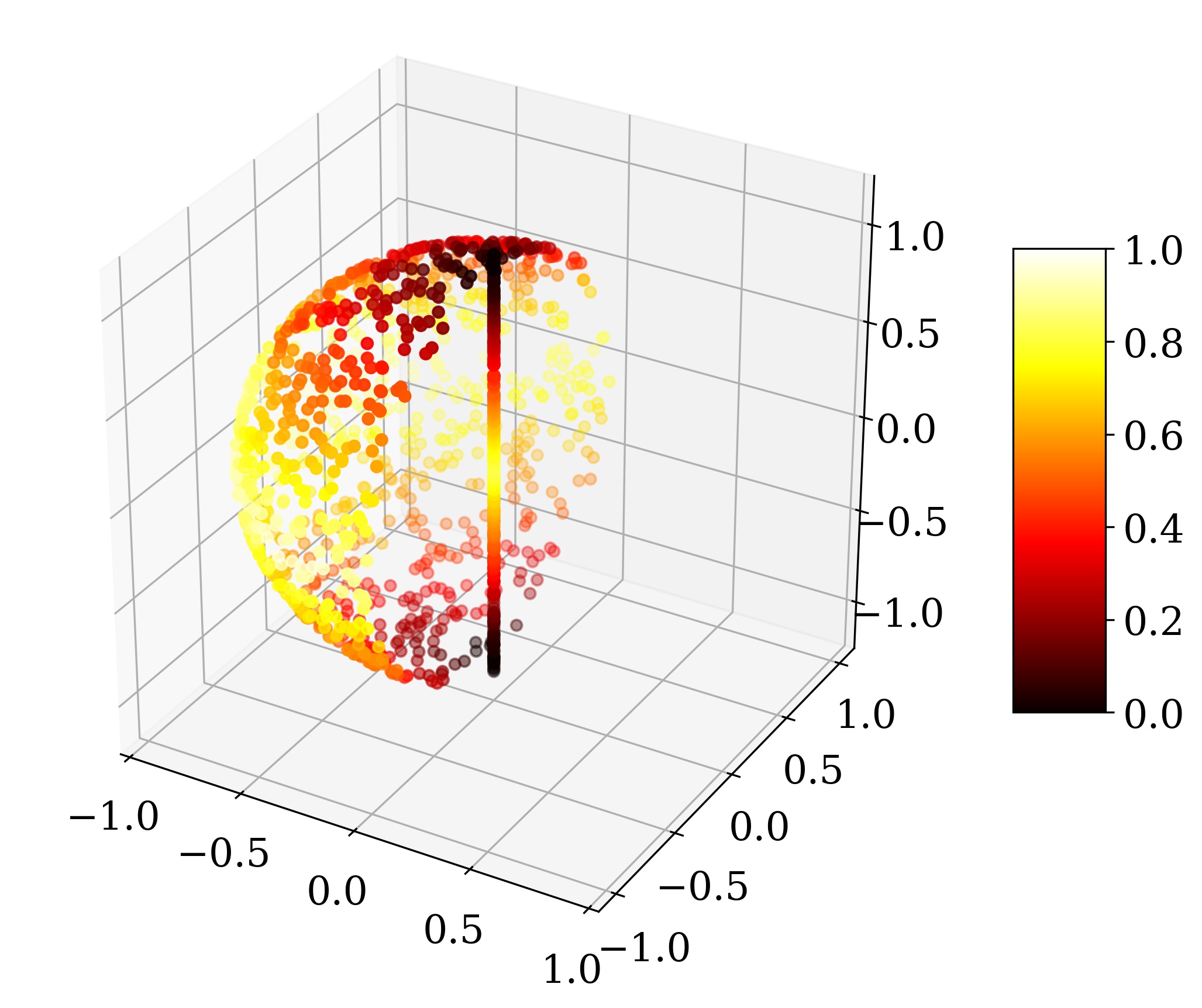}
		\caption{}
		\label{fig:depth_sphereline_cut}
	\end{subfigure}
	\caption{The figure shows examples of the eikonal depth of some densities on manifolds. Figure \ref{fig:depth_sphereline_cut} shows only the points in \ref{fig:depth_sphereline} with $x \leq 0$.}
	\label{fig:manifolds}
\end{figure}

\textbf{Example 2: Densities on manifolds.} In Figure \ref{fig:manifolds} we demonstrate the application of this method to a density living on lower-dimensional structure in three dimensions. In the first, we consider a non-uniform density on a cylinder, with density $\rho(\theta,z)=1-0.9 \sin^2 (\frac{\pi}{2} + 0.5 \pi z + 3 \theta)$, where the points on the cylinder are parametrized as $(\cos \theta,\sin \theta, z)$ for $\theta \in [0, 2 \pi], z \in [-\epsilon,1+\epsilon	]$. We then drew $65^2=4225$ samples in the domain, and constructed a geometric graph using the scaling law described in Section \ref{sec:graph_setting} and the analytical formula for the density. By doing so, our numerical method does not require a priori knowledge of the underlying manifold. We then solved the Eikonal equation using the fast marching method on that geometric graph (see e.g. \cite{sethian2000fast,fadili2021limits}). We imposed a zero boundary condition on the boundary of the manifold (namely the ``edges'' of the cylinder). In practice one could seek to automatically detect the boundary points of the manifold (see e.g. \cite{calder2021boundary}), but for simplicity we chose to specify it directly. 

In the second part of Figure \ref{fig:manifolds} we considered the union of two manifolds, namely the unit sphere along with the line $z = (0,0,t)$ for $t \in [-1-\epsilon,1+\epsilon]$, where the densities on each manifold are taken to be uniform and the two manifolds have equal probability weight. This is an example of a metric measure space, and is a situation where the dimension of the different components is not constant. Indeed, at the point where the line and the sphere intersect the structure of the union of the two manifolds ceases to be Riemannian. Such unions of manifolds have been considered in the context of spectral methods in \cite{trillos2021large}. Again here we take random samples from the two manifolds and then construct geometric graphs using the same scaling laws. In this case we imposed the boundary condition on the endpoints of the line, which can be seen as the "boundary" of the metric measure space.


\textbf{Example 3: MNIST} We also applied the algorithm to compute statistical depths on a subset of the MNIST data set. This data set, which is a common benchmark in machine learning, consists of 70k, $28 \times 28$ grayscale images of handwritten numbers. We focused on computing the depth corresponding to the set of handwritten $4$'s. In doing so, we constructed a weighted neighborhood graph by first constructing a $k$ nearest neighbors graph $k=10$, and then assigning weights proportional to
\begin{equation}
	w_{ij}= \exp \left(\frac{- 4 |x_i - x_j|^2}{d_k(x_i)^2} \right),
\end{equation}
where $d_K(x_i)$ is the distance from $x_i$ to its $k$th-nearest neighbor. We consider that the probabily density is uniformly distributed on all the points in the MNIST set. The depth was then calculated by using the difference scheme in \eqref{eq:differenceScheme}. The boundary of the set of handwritten $4$'s was taken to be the set of all other digits within the data set. This then allows us to assign a depth value to every element of the set of $4$'s. In order to visualize, we display in Figure \ref{fig:MNIST} a random sample of the digits plotted according to their respective depth values. This allows us to identify outlying and inlying digits, and their relative frequencies. This example demonstrates the applicability of these depths to both high-dimensional and graph-based data.  We remark here that similar results have been independently and concurrently observed in \cite{calder2021boundary}, which was focused on boundary estimation and gave a path distance as an example of their method.

\begin{figure}
	\centering
	\includegraphics[width=0.8\textwidth]{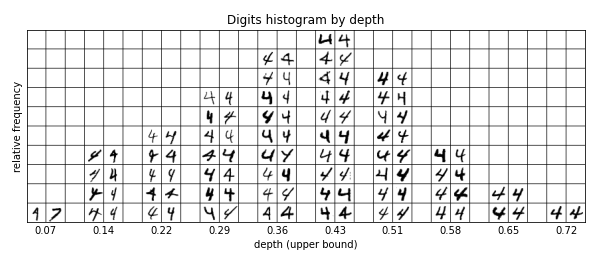}
	\caption{\textcolor{\changed}{The figure shows some representative points belonging to different level sets of the eikonal depth on the elements in MNIST labeled as 4's. The number of representatives in each depth range indicates the proportion of 4's in each range. We can observe that points lying deep in the distribution seem to be neat examples of a hand-written '4'.}}
	\label{fig:MNIST}
\end{figure} 

\textbf{Implementation details:} All code was implemented in Python, and ran in less than a minute on a modest desktop computer. The graph construction for MNIST utilized the Python package ``GraphLearning" \cite{calder2020poisson}. We expect that most of the code would run in seconds if implemented in a lower level language. 
}

	\newpage

\appendix
\section{Proofs of Some Results}
\textcolor{\changed}{
	\begin{proposition}
		\label{prop:comparisonbounded}
		Consider the equation 
		\begin{equation}
			|D u (x)| = \rho (x)
		\end{equation}
		where $\rho$ is an integrable function with support equal to $\mathbb{R}^d$. Additionally, let us suppose that $u \rightarrow 0$ as $|x| \rightarrow \infty$. Let $u$ and $v$ be respectively a bounded subsolution and a bounded supersolution of the equation. Then we have that $u \leq v$ in all of $\mathbb{R}^d$.
	\end{proposition}
}

\begin{proof}[Proof of Proposition \ref{prop:comparisonbounded}]
	Consider $u,v$ bounded sub- and super- solutions. The comparison principle follows from an argument by contradiction. Let us first note that the function $u_{\epsilon}=(1-\epsilon)$, for a given $\epsilon>0$, is a strict subsolution in the viscosity sense, that is, if $|D u (x)| \leq \rho (x)$, then $|D u_{\epsilon} (x)| + \epsilon \rho(x) \leq \rho(x)$.

	Let us define, for $\alpha>0$, the auxiliary function
	\begin{equation}
	\Phi(x,y)= u_{\epsilon}(x) - v(y) - \frac{\alpha}{2} |x-y|^2
	\end{equation}
	and suppose, contrary to the statement of the comparison principle, that $M:= \sup_{\mathbb{R}^d} (u_{\epsilon}-v) >0$. This implies that $\sup_{\mathbb{R}^d \times \mathbb{R}^d} \Phi(x,y) >0$, as 
	$\sup_{\mathbb{R}^d \times \mathbb{R}^d} \Phi(x,y) \geq \sup_{\mathbb{R}^d} (u_{\epsilon}-v) >0$.
	We can take a sequence $\{(x_k,y_k)\}$ such that  $\limsup_{k \infty} \Phi(x_k,y_k)=M$. Because $\Phi(x_k,y_k)>0$ for $k$ large enough, we have
	\begin{equation}
	\frac{\alpha}{2} |x_k -y_k| = u_{\epsilon}(x_k) - v (y_k) - \Phi(x_k,y_k) \leq u_{\epsilon}(x_k) -v(y_k) .
	\end{equation}
	And taking into account that $u$ is bounded, then $\frac{1}{2} |x_k - y_k|^2 \leq \frac{M}{\alpha}$ for some $M>0$. This implies that $x_k$ and $y_k$ remain close for increasingly large values of $k$. Moreover, increasing the value of the parameter $\alpha$ only reduces the distance between $x_k$ and $y_k$. On the other hand, we have that
	\begin{equation}
	\Phi(x_k,y_k) \leq u_{\epsilon}(x_k)- v(y_k) - \frac{\alpha}{2} |x_k - y_k|^2 \leq u_{\epsilon}(x_k) - v (y_k)
	\end{equation}
	After taking limits this inequality yields
	\begin{equation}
	M=\limsup_{k \rightarrow \infty} \Phi(x_k,y_k) \leq \limsup_{k \rightarrow \infty} \{u_{\epsilon}(x_k) - v (y_k)\}
	\end{equation}
	Note that by the above bound on $|x_k -y_k|$ we can only have two situations, either both $\{x_k\}$ and $\{y_k\}$ remain bounded, or both go to infinity as $k \rightarrow \infty$. As both $u_{\epsilon}(x)$ and $v(x)$ approach zero as $x$ goes to infinity by assumption, we cannot have that both  $\{x_k\}$ and $\{y_k\}$ escape to infinity (this would imply that $M \leq 0$). Given that $\{x_k\}$ and $\{y_k\}$ are bounded for all $k$ we can construct a subsequence that converges to a point $(x_{\alpha}, y_{\alpha})$ at which the supremum of $\Phi(x,y)$ is realized, that is,
	\begin{equation}
	M= \Phi(x_{\alpha},y_{\alpha}) = u_{\epsilon}(x_{\alpha}) - v(y_{\alpha}) - \frac{\alpha}{2} |x_{\alpha} -y_{\alpha}|^2 .
	\end{equation}
	
	This means that $\varphi(x)= v(y_{\alpha}) +\frac{\alpha}{2} |x - y_{\alpha}|^2$ is a smooth function such that $u_{\epsilon}-\varphi$ attains a maximum at $x_{\alpha}$ and that $\psi(y)= u_{\epsilon} (x_{\alpha}) - \frac{\alpha}{2} |x_{\alpha} - y|^2$ is a smooth function such that $v- \psi$ attains a minimum at $y_{\alpha}$. We then have that
	\begin{equation}
	\begin{array}{ccc}
	|D \varphi (x_{\alpha})| - \rho(x_{\alpha}) & \leq & \epsilon \rho(x_{\alpha}) \\
	|D \psi (y_{\alpha})| - \rho (y_{\alpha}) & \geq & 0
	\end{array}
	\end{equation}
	and therefore
	\begin{equation}
	\begin{array}{ccl}
	\epsilon \rho(x_{\alpha}) & \leq & (|D \psi(y_{\alpha})|- \rho(y_{\alpha}))- (|D \varphi (x_\alpha)|- \rho (x_{\alpha})) \\
	& = & (\alpha |x_{\alpha}- y_{\alpha}|^2 - \rho(y_{\alpha}))- (\alpha |x_{\alpha}- y_{\alpha}|^2 -\rho(x_{\alpha})) \\
	& = & \rho(x_{\alpha}) - \rho(y_{\alpha}).
	\end{array}
	\end{equation}
	
	The continuity of $\rho$, the fact that we have $|x_{\alpha}-y_{\alpha}| \leq \frac{M}{\alpha}$ and the positivity of $\epsilon \rho(x_{\alpha})$ yields a contradiction. We the have $u_{\epsilon} \leq v$ on $\mathbb{R}^d$. By construction, we also have that $u \leq u_{\epsilon}$ and then the comparison principle follows.
\end{proof}

{\color{\changed}
\begin{proof}[Proof of Proposition \ref{prop:comparisonunbounded}]
	For $R>1$, let us construct the auxiliary function $$w(x)= (R-1) \Psi \left(\frac{u(x)}{R} \right)$$ where $\Psi$ is a globally bounded smooth function such that $\Psi(y)=y$ for $|y| \leq R$. The function $w$ is a strict subsolution in the sense that if $\psi$ is a smooth function such that $\Psi-\psi$ attains a maximum at $x$ then
	\begin{equation}
	|D\psi(x)|+\frac{\rho(x)}{R} \leq \rho(x)
	\end{equation}
	in the sense of viscosity solutions.
	
	Because now both $w$ and $v$ are bounded, we can use the same proof as in the previous proposition to show that $w \leq v$. It only remains to notice that from the definition of $w$ we have that $u \leq v$ if $|x| \leq R$ and this holds for any value of $R$. We then have that $u \leq v$.
\end{proof}
}

\bibliographystyle{plain}
\bibliography{mybib}

\end{document}